\documentclass[12pt]{amsart}
\usepackage{amssymb}
\usepackage{amsmath, amsfonts, amsthm, amssymb, cite}
\DeclareMathOperator{\Tr}{Tr}
\begin{document}


\newtheorem{theorem}{Theorem}[section]
\newcommand{\mar}[1]{{\marginpar{\textsf{#1}}}}
\numberwithin{equation}{section}
\newtheorem*{theorem*}{Theorem}
\newtheorem{prop}[theorem]{Proposition}
\newtheorem*{prop*}{Proposition}
\newtheorem{lemma}[theorem]{Lemma}
\newtheorem{corollary}[theorem]{Corollary}
\newtheorem*{conj*}{Conjecture}
\newtheorem*{corollary*}{Corollary}
\newtheorem{definition}[theorem]{Definition}
\newtheorem*{definition*}{Definition}
\newtheorem{remarks}[theorem]{Remarks}
\newtheorem*{remarks*}{remarks}

\newtheorem*{not*}{Notation}

\newcommand{\Dslash}{{D\mkern-11.5mu/\,}}

\newcommand\pa{\partial}
\newcommand\cohom{\operatorname{H}}
\newcommand\Td{\operatorname{Td}}
\newcommand\Trig{\operatorname{Trig}}
\newcommand\Hom{\operatorname{Hom}}
\newcommand\End{\operatorname{End}}
\newcommand\Ker{\operatorname{Ker}}
\newcommand\Ind{\operatorname{Ind}}
\newcommand\oH{\operatorname{H}}
\newcommand\oK{\operatorname{K}}
\newcommand\codim{\operatorname{codim}}
\newcommand\Exp{\operatorname{Exp}}
\newcommand\CAP{{\mathcal AP}}
\newcommand\T{\mathbb T}
\newcommand{\M}{\mathcal{M}}
\newcommand\ep{\epsilon}
\newcommand\te{\tilde e}
\newcommand\Dd{{\mathcal D}}

\newcommand\what{\widehat}
\newcommand\wtit{\widetilde}
\newcommand\mfS{{\mathfrak S}}
\newcommand\cA{{\mathcal A}}
\newcommand\maA{{\mathcal A}}
\newcommand\maF{{\mathcal F}}
\newcommand\maN{{\mathcal N}}
\newcommand\cM{{\mathcal M}}
\newcommand\maE{{\mathcal E}}
\newcommand\cF{{\mathcal F}}
\newcommand\maG{{\mathcal G}}
\newcommand\cG{{\mathcal G}}
\newcommand\cH{{\mathcal H}}
\newcommand\maH{{\mathcal H}}
\newcommand\cO{{\mathcal O}}
\newcommand\cR{{\mathcal R}}
\newcommand\cS{{\mathcal S}}
\newcommand\cU{{\mathcal U}}
\newcommand\cV{{\mathcal V}}
\newcommand\cX{{\mathcal X}}
\newcommand\cD{{\mathcal D}}
\newcommand\cnn{{\mathcal N}}
\newcommand\wD{\widetilde{D}}
\newcommand\wL{\widetilde{L}}
\newcommand\wM{\widetilde{M}}
\newcommand\wV{\widetilde{V}}
\newcommand\vf{\varphi}
\newcommand\Ee{{\mathcal E}}
\newcommand{\npartial}{\slash\!\!\!\partial}
\newcommand{\Heis}{\operatorname{Heis}}
\newcommand{\Solv}{\operatorname{Solv}}
\newcommand{\Spin}{\operatorname{Spin}}
\newcommand{\SO}{\operatorname{SO}}
\newcommand{\ind}{\operatorname{ind}}
\newcommand{\Index}{\operatorname{index}}
\newcommand{\ch}{\operatorname{ch}}
\newcommand{\rank}{\operatorname{rank}}
\newcommand{\G}{\Gamma}
\newcommand{\HK}{\operatorname{HK}}
\newcommand{\Dix}{\operatorname{Dix}}
\newcommand{\hideqed}{\renewcommand{\qed}{}}
\newcommand{\abs}[1]{\lvert#1\rvert}
 \newcommand{\A}{{\mathcal A}}
        \newcommand{\D}{{\mathcal D}}
\newcommand{\HH}{{\mathcal H}}
        \newcommand{\LL}{{\mathcal L}}
        \newcommand{\B}{{\mathcal B}}
        \newcommand{\K}{{\mathcal K}}
\newcommand{\oo}{{\mathcal O}}
         \newcommand{\PP}{{\mathcal P}}
        \newcommand{\s}{\sigma}
        \newcommand{\coker}{{\mbox coker}}
        \newcommand{\p}{\partial}
        \newcommand{\dd}{|\D|}
        \newcommand{\n}{\parallel}
\newcommand{\bma}{\left(\begin{array}{cc}}
\newcommand{\ema}{\end{array}\right)}
\newcommand{\bca}{\left(\begin{array}{c}}
\newcommand{\eca}{\end{array}\right)}
\newcommand{\sr}{\stackrel}
\newcommand{\da}{\downarrow}
\newcommand{\tD}{\tilde{\D}}
               \newcommand{\h}{\mathbf H}
\newcommand{\Z}{\mathcal Z}
\newcommand{\N}{\mathbf N}
\newcommand{\tto}{\longrightarrow}
\newcommand{\ZZ}{{\mathcal Z}}
\newcommand{\ben}{\begin{displaymath}}
        \newcommand{\een}{\end{displaymath}}
\newcommand{\be}{\begin{equation}}
\newcommand{\ee}{\end{equation}}
        \newcommand{\bean}{\begin{eqnarray*}}
        \newcommand{\eean}{\end{eqnarray*}}
\newcommand{\nno}{\nonumber\\}
\newcommand{\bea}{\begin{eqnarray}}
        \newcommand{\eea}{\end{eqnarray}}
\newcommand{\x}{\times}

\newcommand{\Ga}{\Gamma}
\newcommand{\e}{\epsilon}
\renewcommand{\L}{\mathcal{L}}
\newcommand{\supp}[1]{\operatorname{#1}}
\newcommand{\norm}[1]{\parallel\, #1\, \parallel}
\newcommand{\ip}[2]{\langle #1,#2\rangle}
\newcommand{\nc}{\newcommand}
\nc{\gf}[2]{\genfrac{}{}{0pt}{}{#1}{#2}}
\nc{\mb}[1]{{\mbox{$ #1 $}}}
\nc{\real}{{\mathbb R}}
\nc{\R}{{\mathbb R}}
\nc{\C}{{\mathbb C}}
\nc{\comp}{{\mathbb C}}
\nc{\ints}{{\mathbb Z}}
\nc{\Ltoo}{\mb{L^2({\mathbf H})}}
\nc{\rtoo}{\mb{{\mathbf R}^2}}
\nc{\slr}{{\mathbf {SL}}(2,\real)}
\nc{\slz}{{\mathbf {SL}}(2,\ints)}
\nc{\su}{{\mathbf {SU}}(1,1)}
\nc{\so}{{\mathbf {SO}}}
\nc{\hyp}{{\mathbb H}}
\nc{\disc}{{\mathbf D}}
\nc{\torus}{{\mathbb T}}
\newcommand{\tk}{\widetilde{K}}
\newcommand{\boe}{{\bf e}}\newcommand{\bt}{{\bf t}}
\newcommand{\vth}{\vartheta}
\newcommand{\CGh}{\widetilde{\CG}}
\newcommand{\db}{\overline{\partial}}
\newcommand{\tE}{\widetilde{E}}
\newcommand{\tr}{\mbox{tr}}
\newcommand{\ta}{\widetilde{\alpha}}
\newcommand{\tb}{\widetilde{\beta}}
\newcommand{\txi}{\widetilde{\xi}}
\newcommand{\hV}{\hat{V}}
\newcommand{\IC}{\mathbf{C}}
\newcommand{\IZ}{\mathbf{Z}}
\newcommand{\IP}{\mathbf{P}}
\newcommand{\IR}{\mathbf{R}}
\newcommand{\IH}{\mathbf{H}}
\newcommand{\IG}{\mathbf{G}}
\newcommand{\IS}{\mathbf{S}}
\newcommand{\CC}{{\mathcal C}}
\newcommand{\CD}{{\mathcal D}}
\newcommand{\CS}{{\mathcal S}}
\newcommand{\CG}{{\mathcal G}}
\newcommand{\CL}{{\mathcal L}}
\newcommand{\CO}{{\mathcal O}}
\nc{\ca}{{\mathcal A}}
\nc{\cag}{{{\mathcal A}^\Gamma}}
\nc{\cg}{{\mathcal G}}
\nc{\chh}{{\mathcal H}}
\nc{\ck}{{\mathcal B}}
\nc{\cd}{{\mathcal D}}
\nc{\cl}{{\mathcal L}}
\nc{\cm}{{\mathcal M}}
\nc{\cn}{{\mathcal N}}
\nc{\cs}{{\mathcal S}}
\nc{\cz}{{\mathcal Z}}
\nc{\sind}{\sigma{\rm -ind}}

\newcommand\clFN{{\mathcal F_\tau(\mathcal N)}}       
\newcommand\clKN{{\mathcal K_\tau(\mathcal N)}}       
\newcommand\clQN{{\mathcal Q_\tau(\mathcal N)}}       %
\newcommand\tF{\tilde F}
\newcommand\clA{\mathcal A}
\newcommand\clH{\mathcal H}
\newcommand\clN{\mathcal N}
\newcommand\Del{\Delta}
\newcommand\g{\gamma}
\newcommand\eps{\varepsilon}
\newcommand{\dslash}{{\partial\mkern-10mu/\,}}

\newcommand{\sepword}[1]{\quad\mbox{#1}\quad} 
\newcommand{\comment}[1]{\textsf{#1}}   

 \title{Integration on locally compact noncommutative spaces}

\author{A. L. CAREY}
\address{Mathematical Sciences Institute,
 Australian National University\\
 Canberra ACT, 0200 AUSTRALIA\\
 e-mail: alan.carey@anu.edu.au\\}
\author{V. GAYRAL}
\address{Laboratoire de Mathematiques\\
Universite Reims Champagne-Ardenne\\
Moulin de la Housse-BP 1039, 51687 Reims, FRANCE\\
e-mail: victor.gayral@univ-reims.fr}
\author{A. RENNIE}
\address{Mathematical Sciences Institute\\
Australian National University\\
Canberra ACT, 0200 AUSTRALIA\\
e-mail: adam.rennie@anu.edu.au}
\author{F. A. SUKOCHEV}
\address{School of Mathematics and Statistics,
University of New South Wales\\
Kensington NSW, 2052 AUSTRALIA\\
e-mail: f.sukochev@unsw.edu.au}

\begin{abstract}
{We present an ab initio approach to integration theory for nonunital spectral triples. This
is done without reference to local units and in the full generality of semifinite noncommutative
geometry. The main result is an equality between the Dixmier trace and generalised residue
of the zeta function and heat kernel  of suitable operators. We also examine definitions for integrable
bounded elements of a spectral triple based on zeta function, heat kernel and Dixmier
trace techniques. We show that zeta functions and heat kernels yield equivalent notions
of integrability, which imply Dixmier traceability.}
\end{abstract}
        \maketitle

\section{introduction}
The definition  of nonunital spectral triples
  is dictated by Kasparov theory, however the 
  correct method of specifying summability has not been established on a firm base.
  In particular we need a better understanding of the relationship
  between the residues of the zeta function, asymptotics of the heat kernel and their relationship
  with the Dixmier trace. Previous treatments of this question
  rely on a quasi-localness assumption that allows one to reduce 
  to a `compactly supported case', a notion which may or may not be generally applicable,
  \cite{GGISV,R1,R2}. 
  Here we seek a more general point of view that
  avoids local units. Moreover we treat the problem in the general setting of semifinite noncommutative geometry.

It is known
by Connes' trace theorem, \cite{Co1}, that for a {\em compact} Riemannian 
manifold $M$ of dimension $p$, there is an intimate
connection between the residue of the zeta function of the
Laplacian at its first singularity, the Dixmier trace, and
integration theory of functions  with respect to the standard measure.  
More precisely,
we let $\Delta$ be a Laplace type operator acting on sections of a vector
bundle and $f\in C^\infty(M)$ be a smooth function acting by
multiplication on smooth sections. Then $\Delta$ extends to an essentially
self-adjoint operator  with compact resolvent on $L^2$-sections
and the multiplication operator $M_f$ extends to a
bounded operator. Then if
$G:=(1+\Delta)^{-p/2}$, the operator $M_fG$ has singular numbers 
$\mu_n=O(\frac{1}{n})$, and the Dixmier trace of this operator is equal to 
(up to a constant depending only on the dimension $p$) the
integral of $f$ over $M$ with respect to the measure in
$M$ given by the volume form associated to a choice of a Riemannian metric. 
The residue of the zeta function 
$\Tr(M_fG^s)$ at $s=1$ and the    rescaled limit of $\Tr(M_fe^{-t\Delta})$ when $t\to 0^+$, also coincide (up to a constant) with the integral of $f$.

More recently in \cite{LPS, LS} it has been shown that
for $f\in L^1(M)$ we can define the zeta function $\Tr(G^{s/2}M_f G^{s/2})$
for $s>1$ and that the (generalised or $\omega$) residue at $s =1$ is equal to
the integral of $f$ over $M$, {\em while
the Dixmier trace of $G^{1/2}M_f G^{1/2}$ does not exist for all $f$ in $L^1(M)$}.
Related results which indicate the difficulties of noncommutative integration are that
the Dixmier
trace of $M_fG$  exists if and only if $f\in L^2(M)$. In summary, the 
generalised residue of the zeta function exists in greater generality than the Dixmier
trace, and recovers integration on the manifold. When the Dixmier trace also exists, it
agrees with the zeta residue (up to a constant). These results clearly indicate the 
subtleties 
one has to face when manipulating products, symmetrised or unsymmetrised, 
even in the 
compact manifold case.

When the manifold is not compact the situation is much less clear.
Moreover the analogous noncommutative integration theory (for
nonunital pre-$C^*$-algebras) has not been developed from first
principles, rather, an assumption is made about the existence of a
system of local units for the algebra, \cite{R1,R2}. This is a plausible
assumption given that there is a  local structure on a noncompact
manifold $M$. Moreover this assumption allows one to show that various 
possible definitions agree, and has been used to prove results
analogous to the unital case, such as
relating the zeta function and  Dixmier trace in that setting, \cite{GGISV,R1,R2}. Noncommutative examples
which fit into this framework are   described in \cite{GGISV,PRe,PRS,R1}.

Thus we
ask what can we   learn about noncommutative integration by
beginning with the zeta function, rather than the Dixmier trace, in
the nonunital case. More precisely we phrase this, our
fundamental  question, as follows: {\em what is the appropriate
noncommutative integral for nonunital spectral triples?}

There are various possible answers to our fundamental question
and we aim, in this note, to explain one, which in view of the results we obtain,
is potentially the most natural.
A second question, which we believe our approach
helps to understand,  is {\em what summability
hypotheses are needed to  prove the local index formula for nonunital
spectral triples}. We shall address this question in another place.
Thus for the moment our focus is on whether there are reasonable notions of
integration, integrability, and most fundamentally, spectral dimension
in the non\-unital case. Importantly,
 we will couch our results, and choose proofs, that go through for the general
framework of semifinite spectral triples.

To explain this further we need to resolve a notational issue. Let
$\cn$ be a semifinite von Neumann algebra with fixed faithful
normal semifinite trace $\tau$. In \cite{CRSS} we introduced a
family of ideals 
that we denoted by
${\mathcal Z}_p:= {\mathcal Z}_p(\cn, \tau)$. These ideals are
naturally defined by the asymptotics of the zeta function 
$s\mapsto \tau(T^s)$ as $s$ converges from above
to the infimum of values for which this trace is finite. As a consequence of our
analysis we found that ${\mathcal Z}_1$ coincided with the (dual
to the) Macaev ideal for which it has become standard in
noncommutative geometry to use the notation $\cl^{1,\infty}$.
In this paper we will use the notation ${\mathcal Z}_p$ even in
the case $p=1$ for consistency. We make more comments on this matter later.

Now, the problem that arises for noncommutative and nonunital integration 
is that we are dealing with products of operators, neither
of which individually lies in ${\mathcal Z}_1$ but whose product does lie in
${\mathcal Z}_1$. Examples show that a similar phenomenon persists in the case of
noncommutative algebras \cite{GGISV,PRe,PRS}
and general semifinite traces. The difficulties posed by this situation for the analysis of
Schr\"odinger  operators has been explained in detail in \cite[Chapter $f(X)g(-i\nabla)$]{Simon}.

The aim of this article is to study the
general situation from a point of view as close as possible to that of
the unital case without assuming the existence of a quasi-local structure.

We begin with some preliminary facts in Sections 2 and 3. 
One of our first results in Section 4 is to establish is that we could just as well 
study the heat
kernel asymptotics. That is,  the heat kernel asymptotics imply precisely the same 
data and exist in the same generality
as the generalised zeta residue.
 We next show that for bounded operators $a\in\cn$, the existence of the 
 generalised $\zeta$-residue 
$\omega-\lim_{r\to\infty} \frac{1}{r}\Tr(a^*G^{1+\frac{1}{r}}a)$ 
implies that $a^*Ga\in\Z_1$. A range of similar statements
shows the compatibility of the zeta residue with the ideals $\Z_p$ and $\LL^{p,\infty}$.
We use implicitly a Hilbert algebra framework (described in Section 4), that is, our approach yields a noncommutative $L^2$ theory.

Our main result is proved in Section 5, namely 
that under `suitable conditions', the zeta function or the heat kernel 
can be used to compute
the Dixmier trace. We seek the weakest hypotheses feasible for
this main result with the consequence that
 the proof is surprisingly subtle and lengthy. We finish in Section 7 by showing that there
are definitions of spectral triple and spectral dimension in the nonunital case for which these `suitable conditions' may be verified.

\subsection*{ Acknowledgements} We thank A. Bikchentaev, A. Connes,
N. Kalton, R. Nest and A. Sedaev for 
discussions and useful references, and D. Potapov for a careful reading of the manuscript. 

\section{The ideals $\Z_p$}
 \label{sec:unital-case}
\subsection{Definitions and notations}
We fix a semifinite von Neumann
algebra $\cn$ acting on a separable Hilbert space $\HH$. Fix also
a faithful, semifinite, normal trace $\tau:\cn\to\C$. The zeta
function of a positive $\tau$-compact operator $T$ is given by
$\zeta(s)=\tau(T^s)$ for real positive $s$, on the assumption that
there exists some $s_0$ for which the trace is finite. Note that
it is then true that $\tau(T^s)<\infty$ for all $s\geq s_0$.

Let us consider the space $\Z_1$ introduced in \cite{CRSS}:
$$
\Z_1:={\mathcal Z}_1(\clN,\tau):=\big\{T\in {\mathcal N}:\ \|T\|_{{\mathcal
Z}_1}=\limsup_{p\searrow 1}(p-1)\tau(|T|^p)<\infty\big\}.
$$
Note that $||.||_{\Z_1}$ is a seminorm, only. The next  equality is easy to see 
(for details consult \cite{CRSS})
$$
\|T\|_{{\mathcal Z}_1}
=\limsup_{p\searrow 1}(p-1)\Big(\int_0^\infty \mu_t(|T|)^pdt\Big)^{1/p}
=\limsup_{p\searrow1}(p-1)\|T\|_{p}.
$$ 
(We use the notation
$\L^p$ for the Schatten ideals in $(\clN,\tau)$ and $\|.\|_p$ for the Schatten norms.) 
More generally, we define for $p\geq 1$ the spaces $\Z_p$, the $p$-convexifications of
$\Z_1$ \cite{LinTza1979-MR540367}, by
$$
{\mathcal Z}_p(\cn,\tau)=\big\{T\in {\mathcal N}:\ \|T\|_{{\mathcal Z}_p}=
\limsup_{q\searrow p}\big((q-p)\tau(|T|^q)\big)^{1/q}<\infty\big\}.
$$

An important role in noncommutative geometry is played by the ideal
of compact operators whose partial sums of singular values are
logarithmically divergent. This ideal (in the setting of
general semifinite von Neumann algebras) is probably best
expressed through the terminology of noncommutative
Marcinkiewicz spaces and we refer to \cite{LSS,CS,CRSS} for a
detailed exposition of relevant parts of this theory. Here, we set
$$
M_{1,\infty}(\clN,\tau):=\big\{
T\in {\mathcal N}:\ \|T\|_{1,\infty}:=\sup_{0<t<\infty}
\log(1+t)^{-1} \int_0^t \mu_s(T)ds < \infty\big\}.
$$ 
We will usually take $(\cn,\tau)$ as fixed, and write $M_{1,\infty}$ instead of 
$M_{1,\infty}(\cn,\tau)$, as this will cause no confusion. Similar comments apply to 
the notation for other
ideals.
 
The Banach space
$(M_{1,\infty},\|.\|_{1,\infty})$ was probably
first considered by Matsaev \cite{Mat}. It may be  viewed as a
noncommutative analogue of a Sargent (sequence) space, see
\cite{Sar}. In noncom\-mutative geometry it has become customary to
use the notation $\mathcal{L}^{1,\infty}$ to denote the ideal
$M_{1,\infty}$. However we will avoid the $\L^{1,\infty}$
notation as it clashes with the well-established notation of
quasi-normed weak $L_1$-spaces. For a fuller treatment of the
history of the space $M_{1,\infty}$ and additional references, we
refer the interested reader to the recent paper \cite{Pie} by Pietsch.

More generally, we let $M_{p,\infty}$, $p\geq 1$, denote the $p$-convexification
of the space $M_{1,\infty}$, defined by
\begin{equation}
\label{Mp}
M_{p,\infty}(\clN,\tau):=\big\{
T\in {\mathcal N}:\ \|T\|_{p,\infty}^p:=\sup_{0<t<\infty}
\log(1+t)^{-1} \int_0^t \mu_s(|T|^p)ds < \infty\big\}.
\end{equation}

In our present context, it is important to observe that it follows
from \cite[Theorem 4.5]{CRSS} that the sets
$M_{1,\infty}$ and ${\mathcal Z}_1$ coincide
and that $\|T\|_0 \leq e\|G\|_{Z_1}$ and $\|T\|_{Z_1} \leq
\|T\|_{1,\infty}$, where the seminorm~$\left\|\cdot\right\|_0$
is the distance in the norm $\left\|\cdot\right\|_{1,\infty}$ to the subspace
$M^0_{1,\infty}$ of $M_{1,\infty}$ formed by
the $(1,\infty)$-norm closure of the trace ideal $\CL^1\subset
M_{1,\infty}$. To be consistent with our $\Z_p$ notation, 
we also denote the latter ideal as $\Z_1^0$.
Of course, the spaces
$\Z_p$ and $M_{p,\infty}$ coincide. We also stress that $\L^{p,\infty}$, $p>1$, the 
collection of $\tau$-compact operators for which $\mu_t(T)=O(t^{-1/p})$, are strictly 
included in $\Z_p$, \cite{CRSS}. Moreover, a careful  inspection of its proof, gives the following strengthening  of Theorem 4.5 in \cite{CRSS}:
\begin{theorem}
\label{great}
The norm $\|.\|_{1,\infty}$ of the Marcinkiewicz space $M_{1,\infty}$ is equivalent to the $\zeta$-norm:
 $$\sup_{p>1}\,(p-1)\|T\|_p\,,\quad T\in\cn.$$
\end{theorem}

Another important feature of the ideals ${\mathcal
Z}_1=M_{1,\infty}$ is that they support
singular traces. Let $M_{1,\infty}(\cH)$ denote
$M_{1,\infty}$ when $(\clN,\tau)$ is given by the
algebra of all bounded linear operators equipped with standard
trace. In \cite{Dix1}, J.~Dixmier constructed a non-normal
semifinite trace living on the ideal $M_{1,\infty}(\cH)$ using the
weight
$$
\Tr_\omega(T) := \omega \Big( \Big\{ \frac{1}{\log(1+k)}
\sum_{j=1}^k \mu_j(T)\Big\}_{k=1}^\infty  \Big) \ , \ T\geq0,
$$
associated to a translation and dilation invariant state $\omega$
on $\ell^{\infty}$.  The
seminorm $\left\|\cdot\right\|_{{\mathcal Z}_1}$ and all Dixmier traces
$\Tr_\omega$ vanish on $M^0_{1,\infty}(\cH)$ and this provides a
first (albeit tenuous) connection between Dixmier traces and zeta
functions. This connection  runs much deeper however, and will be
explained further in various parts of the present manuscript.

\subsection{The Calderon-Lozanovski\u{\i} construction}
\label{CL}

Let $(X,\Sigma, \mu)$ be a complete $\sigma$-finite measure space. By 
$L_{0}(X, \mu)$ we will denote the set of all measurable functions which are 
finite a.e.. As usual we will identify functions equal almost everywhere. A linear 
subspace of $L_{0}(X, \mu)$ is called a {\it K\"othe function space} if it is normed 
space which is an order ideal in  $L_{0}(X, \mu)$, i.e. if $f\in E$ and $|g|\leq|f|$ a.e., 
then $g\in E$ and $\|g\|\le \|f\|$. A norm-complete K\"othe function space is called 
a {\it Banach function space}. Given a  K\"othe function space $E$ we have that
\[E^{\times}=\Big\{f\in L_{0}(X,\mu): \int_{X}|fg|\,d\mu <\infty \mbox{ for all }g\in E\Big\}\,,\]
is a Banach function space with the Fatou property, which is called the 
associate (or K\"{o}the dual) space of $E$.
Recall that a Banach function space $Y$ has the Fatou property if for any increasing
positive sequence $(x_n)$ in $Y$  with $\sup_n \|x_n\|_X < \infty$ we 
have that $\sup_n x_n \in Y$
and $\| \sup_n x_n\|_Y = \sup_n \|x_n\|_Y$. This property is equivalent to the 
apparently stronger
property that  whenever $\{x_n\}_{_{n\geq 1}}\subseteq Y$,\ 
$x\in L_0(X,\mu),\ x_n\to x$ almost
everywhere and $\sup _n\Vert x_n\Vert _{_{Y}} <\infty $, we have
$x\in Y$ and $\Vert x\Vert _{_{Y}}\leq \liminf _{n\to \infty}\Vert x_n\Vert _{_{Y}}$.
For every Banach function space $Y$ we have the embedding 
$Y \subseteq Y^{\times\times}$  with $\|x\|_{Y^{\times\times}} \leq \|x\|_{Y}$ 
for any 
$x \in Y$. Moreover, $Y = Y^{\times\times}$ with equality of the norms 
if and only if $Y$
has the Fatou property. For a detailed treatment of Banach function 
spaces we 
refer to \cite{KrePet1982-MR649411,LinTza1979-MR540367}. 

We need the Calderon construction of intermediate 
spaces \cite{Calderon-1966}, 
which was studied and extended by Lozanovski{\u\i}.
Let $E$ and $F$ be Banach function spaces on $(X, \mu)$. Then for 
$1< p<\infty$ and $p'=p/(p-1)$,
the Banach function space $E^{\frac 1p}F^{\frac 1{p^{\prime}}}$ is 
defined as the 
space of all $f\in L_{0}(X,\mu)$ such that 
$|f|=|g|^{\frac 1p}|h|^{\frac 1{p^{\prime}}}$ 
for some $g\in E$ and $h\in F$. The norm on 
$E^{\frac 1p}F^{\frac 1{p^{\prime}}}$ is 
defined by
\begin{align*}
\|f\|_{E^{\frac 1p}F^{\frac 1{p^{\prime}}}}
&=\inf \{\|g\|_{E}^{\frac 1p}\|h\|_{F}^{\frac 1{p^{\prime}}}: |f|=|g|^{\frac 1p}|h|^{\frac 1{p^{\prime}}} \text{ for some }g\in E, h\in F\}.
\end{align*}
It is well-known that $E^{\frac 1p}F^{\frac 1{p^{\prime}}}$ is again a Banach 
function space, moreover it has order continuous norm if at least one of 
$E$ and $F$ has order continuous norm. Also $E^{\frac 1p}F^{\frac 1{p^{\prime}}}$ 
has the Fatou property if both $E$ and $F$ have the Fatou property. The above 
construction contains as a special case the so-called $p$-convexification of $E$ by 
taking $F=L^{\infty}(X,\mu)$. We will denote this space by $E^{\frac 1p}$, since as 
sets $E^{\frac 1p}{L^{\infty}(X,\mu)}^{\frac 1{p^{\prime}}}=E^{\frac 1p}$. Note that the 
norm on $E^{\frac 1p}$ is given by $\||f|^{p}\|_{E}^{\frac 1p}$. 

\subsection{Complex interpolation}
\label{interpolation}

To assist the reader we clarify the construction of the ideals
$\Z_p$ in term of complex and real interpolation in the setting of Banach-lattices. 
For the definition of the two  functors of complex interpolation 
$\bar A_{[\theta]}$ (the first method) and  $\bar A^{[\theta]}$ 
(the second method) defined for an arbitrary Banach couple 
$A=(A_0,A_1)$ we refer to  \cite[pp. 88-90]{BergLofstrom}, \cite{Lozanovskii-1976}. 
In general, the two spaces $\bar A_{[\theta]}$  and  $\bar A^{[\theta]}$ are not 
equal, but always $\bar A_{[\theta]}\subset \bar A^{[\theta]}$ with a norm one 
injection. Usually, the main interest is 
attached to the space $\bar A_{[\theta]}$ and the space $\bar A^{[\theta]}$ is 
more or less only a 
technical tool. However, sometimes we also need the second method as well. 
The next result 
follows immediately from Theorem 1 in \cite{Lozanovskii-1976}.

\begin{theorem}
If the Banach spaces, $A_0$ and $A_1$, have the Fatou property then the 
 closed unit ball in $[A_0,A_1]^{[s]}$ coincides
with the closure in $A_0 +A_1$ of the closed unit ball of $A_0^ {1-s}A_1^s$.
\end{theorem}

Combining this result with the fact that the unit ball of the space 
$A_0^ {1-s}A_1^s$ is closed with respect to convergence in measure
and hence also with respect to the norm convergence in $A_0 +A_1$, 
we arrive at the equality
$[A_0,A_1]^{[s]}=A_0^ {1-s}A_1^s$.

\subsection{Marcinkiewicz function and sequence spaces}
These are the main examples of fully symmetric function  and
sequence spaces. 
Let $\Omega$ denote the set of concave functions $\psi : [0,\infty)
\to [0,\infty)$ such that $\lim_{t \to 0^+} \psi(t) = 0$ and
$\lim_{t \to \infty} \psi(t) = \infty$.
Write $x^*$ for the decreasing rearrangement 
of the function $x$: $x^*$ is the right continuous non-increasing 
function whose distribution function coincides with that of 
$|x|$, (see \cite{KrePet1982-MR649411,LinTza1979-MR540367}). 
Then for $\psi \in \Omega$
define the weighted
mean function
$$
a\left( x,t\right) =\frac{1}{\psi \left(
t\right) } \int_{0}^{t}x^{\ast }\left( s\right) ds\quad t>0.
$$
Denote by
$M(\psi)$ the (Marcinkiewicz) space of measurable functions $x$ on
$[0,\infty)$ such that
\begin{equation}
 \|x\|_{M(\psi)} := \sup_{t
> 0} a\left( x,t\right) = \|{a\left( x,\cdot\right)}\|_\infty
 < \infty.
\end{equation}
We assume in this paper that $\psi(t)=O(t)$ when $t\to 0$, which
is equivalent to the continuous embedding $M(\psi)\subseteq
L^\infty([0,\infty))$.
 The definition of the Marcinkiewicz sequence space
$m(\psi)$ of functions on $\mathbb N$
is similar.

\noindent {\bf Example}. Introduce the following functions
$$
\psi_1(t)=\begin{cases}
t\cdot\log2,&0\le t\le 1\\
\log(1+t),& 1\le t <\infty
\end{cases}
\quad\mbox{and} \quad
\psi_p(t)=\begin{cases}
t,&0\le t \le 1\\
t^{1-\frac{1}{p}},& 1\le t <\infty
\end{cases}\,, \quad p>1.
$$
The spaces
$\L^{1,\infty}=\Z_1$ and $\L^{p,\infty}\subset \Z_p,\ p>1,$
are the Marcinkiewicz spaces $M(\psi_1)$ and $M(\psi_p)$
respectively. 
Now, Lozaniovski{\u\i}'s result described in Section \ref{CL}  and the fact that Marcinkiewicz spaces possess the Fatou property yields 
$$
[M(\psi),L^\infty([0,\infty))]^{[s]}=M(\psi)^ {1-s}.
$$
Thus Calderon's second method of complex interpolation applied to a couple $\big(M(\psi), L^\infty\big)$ (on an arbitrary measure space) yields the
$p$-convexification of $M(\psi)$ (with appropriate $p$).

\subsection{Marcinkiewicz operator spaces}

For a definition and discussion of fully symmetric operator spaces and their important 
subclass, operator Marcinkiewicz spaces, we refer e.g. to \cite{LSS, CS,CRSS}. For 
$\mathcal N$ a semifinite von Neumann algebra equipped with a faithful normal 
semifinite trace $\tau$, let $M(\psi)(\mathcal N,\tau)$ be the corresponding operator 
Marcinkiewicz space. It follows from the results given in 
\cite{Dodds-Dodds-Pagter-Interpolation-1989, DoDode1992-MR1188788} that the 
space $[M(\psi)(\mathcal N,\tau), \mathcal N]^{[s]}$ coincides with the (fully symmetric) 
operator space generated by the space $[M(\psi),L^\infty]^{[s]}$ on the von Neumann 
algebra $\mathcal N$. Taking into account the equality above, we see that 
$[M(\psi)(\mathcal N,\tau), \mathcal N]^{[s]}=M(\psi)^ {1-s}(\mathcal N,\tau)$.
It remains to observe that the space $M(\psi)^ {1-s}(\mathcal N,\tau)$ is precisely the 
space described in \cite{CRSS} as $p$-convexification, which is a noncommutative 
counterpart of the Calderon-Lozanovski{\u\i} construction outlined earlier. Thus, 
applying all of the above to the couple
$(\mathcal Z_1, \mathcal N)$, we obtain the equality
$
[\mathcal Z_1, \mathcal N]^{[1-1/p]}=\mathcal Z_p.
$

\section{The heat kernel definition of $\Z_1$}

Later we will focus on the zeta function definition of the ideal $\Z_1$, but
here consider briefly the heat kernel definition. The reason for doing this is to
highlight the similarities and differences between $\Z_1$ and the weak $\L^1$ ideal.

The weak $\L^1$ ideal, denoted $\L_{1,w}$, is given by
\begin{equation}
\label{weak-L1}
\L_{1,w}:=\L_{1,w}(\cn,\tau):=\big\{T\in\cn: \exists C>0\ \mbox{such that}\ \mu_t(T)\leq C/t\big\}.
\end{equation}
A quasi-norm on $\L_{1,w}$ is given by
$\Vert T\Vert_{1,w}:=\inf\big\{C:\,\mu_t(T)\leq C/t\big\}$.
Clearly $\L_{1,w}\subseteq \Z_1$, and in fact it was shown in \cite{CRSS} that the
inclusion is strict.

The (multiplicative) Cesaro mean, $M$, on $\R^*_+$, is defined
on $f:\R^*_+\to [0,\infty)$ by
$$
\big(Mf\big)(\lambda):=\frac{1}{\log(\lambda)}\int_1^\lambda f(t)\frac{dt}{t}.
$$ 
Then, by \cite[Lemma 5.1]{CRSS}, $T\in\Z_1$ if and only if
$\|Mf_T\|_\infty<\infty$ with
$f_T(\lambda):=\lambda^{-1}\tau(e^{-\lambda^{-1}|T|^{-1}})$.\\
A natural question is whether $f_T$ can be unbounded while $Mf_T$ is bounded. 
If so, is it possible that $MMf_T$ is bounded while $Mf_T$ is unbounded (and so on)?

\begin{prop} Let $ T\in\Z_1$ and $0\leq h:\R^*_+\to [0,\infty)$. Then\\
1) if $MMh$ is bounded, so is $Mh$;\\
2)  the function $f_T$ is unbounded and $Mf_T$ is bounded if and only if $T\in\Z_1$ and
$T\not\in \L_{1,w}$;\\
3) the function $f_T$ is bounded if and only if $T\in\L_{1,w}$.
\end{prop}
\begin{proof} Using the definition and integrating by parts we have
\begin{align*}
\big(MMh\big)(x)&=\frac{1}{\log(x)}\int_1^x\,\frac{1}{\log(y)}\,\int_1^y\,
h(\lambda)\,\frac{d\lambda}{\lambda}\,\frac{dy}{y}\\
&=\frac{\log(\log(x))}{\log(x)}\int_1^x\,h(\lambda)\,\frac{d\lambda}{\lambda}
-\frac{1}{\log(x)}\int_1^x\,\log(\log(\lambda))\,h(\lambda)\,\frac{d\lambda}{\lambda}\\
&=\frac{1}{\log(x)}\int_1^x\,\log\left(\frac{\log(x)}{\log(\lambda)}\right)\,h(\lambda)\,
\frac{d\lambda}{\lambda}.
\end{align*}
Now consider $\big(MMh\big)(x^2)$: 
\begin{align*}
(MMh)(x^2)&=\frac{1}{2\log(x)}\int_1^{x^2}\log\left(\frac{2\log(x)}{\log(\lambda)}\right)\,
h(\lambda)\,\frac{d\lambda}{\lambda}\\
&\geq \frac{1}{2\log(x)}\int_1^{x}\log\left(\frac{2\log(x)}{\log(\lambda)}\right)\,
h(\lambda)\,\frac{d\lambda}{\lambda}
\geq \frac{\log(2)}{2\log(x)}\int_1^{x}\,h(\lambda)\,\frac{d\lambda}{\lambda}=\frac{\log(2)}{2}\,(Mh)(x).
\end{align*}
Hence $\big(MMh\big)(x^2)\geq \frac{\log(2)}{2}\,\big(Mh\big)(x)$ and if $MMh$ is bounded, so too 
is $Mh$ proving part $1)$.

For parts $2)$ and $3)$, since we know that $T\in\Z_1$ if and only if $Mf_T$ is bounded,
it suffices to show that $T\in\L_{1,w}$ if and only if $f_T$ is bounded. So suppose that
$T\in\L_{1,w}$ with $\Vert T\Vert_{1,w}\leq C$ with $C>0$. Then
\begin{align*} 
\sup_\lambda \frac{1}{\lambda}\tau(e^{-\lambda^{-1}|T|^{-1}})
&=\sup_\lambda \frac{1}{\lambda}\int_0^\infty\,e^{-\lambda^{-1}\mu_t(T)^{-1}}\,dt\leq \sup_\lambda \frac{1}{\lambda}\int_0^\infty\,e^{-\lambda^{-1}tC^{-1}}\,dt=C.
\end{align*}
For the converse, we may assume that $\cn$ is  either a type $I$ factor with the standard trace or else that the trace $\tau$ is non-atomic. The general case will then follow by considering the embedding of $\cn$ into $\cn\otimes L^\infty([0,1])$.\\
So, suppose that $\sup_\lambda f_T(\lambda)=C$.
 By \cite{ChS}, in the case when $\cn$ is non-atomic, there exists $\tilde T$ in $\cn$ such that $\mu_t(\tilde T)=\mu_t(T)\chi_{[0,t]}$. 
Define then $\tilde{f}_{T,t}(\lambda):=\frac{1}{\lambda}\tau(e^{-\lambda^{-1}\tilde{T}^{-1}})$.
Then for all $t\geq 0$ we have 
$$
\sup_\lambda \frac{1}{\lambda}\,t\,e^{-\lambda^{-1}\mu_t(T)^{-1}}=
\sup_\lambda\tilde{f}_{T,t}(\lambda)\leq \sup_\lambda f_T(\lambda)=C.
$$
In particular, choosing $\lambda=\mu_t(T)^{-1}$ we find that
$\mu_t(T)\leq Ce/t$. The case where $\cn$ is a type $I$ factor is entirely similar.
\end{proof}

Thus we have the curious situation that if $h$ is unbounded, either one application of 
the Cesaro mean $M$, will produce a bounded function, or if not, then successive applications
of the Cesaro mean will never produce a bounded function. 
Moreover, the boundedness of $f_T$ singles out the ideal that arises in practise, namely
$\L_{1,w}$. The use of Cesaro invariant functionals to produce Dixmier traces has 
enlarged the attention of noncommutative integration theory to $\Z_1$, despite the
fact that $\Z_1$ is unnatural from the point of view of most applications.

The distinction between boundedness and unboundedness of the function $f_T$ is
distinct from the issue of measurability, namely whether the value of a given Dixmier
trace is independent of the choice of 
(Cesaro invariant) functional $\omega$. A sufficient condition 
for measurability is the existence of the limit 
$
\lim_{\lambda\to\infty}(Mf_T)(\lambda),
$
however whether this condition is necessary is not known except in special cases; 
see \cite{LSS}.

\section{Banach algebras for nonunital integration}

Much of the recent motivation for noncommutative integration theories comes 
from spectral triples $(\A,\HH,\D,\cn,\tau)$ and their use in index theory. In Section 
\ref{sec:nonunital-case}, we review the definitions, but here just remind the reader
that the commonest situations are when using the self-adjoint unbounded 
operator $\D$, the trace $\tau$ and suitable $s,\,p,\,t\in\R$, one 
(or more) of  the maps
$$
a\mapsto \tau(a(1+\D^2)^{-s/2}),\quad a\mapsto \tau_\omega(a(1+\D^2)^{-p/2}),
\quad a\mapsto \tau(ae^{-t\D^2})
$$
provides a sensible functional on the algebra $\A$. As 
we will explain, there are close links between these various situations. 

Observe that for $p\geq 1$, the operator $(1+\D^2)^{-p/2}$ is positive, injective
and has norm $\leq 1$. These are the main properties we use and in the following
text consider an operator $G$ with these properties. In Section 
\ref{sec:zeta-and-dixy} we will also need some conditions on commutators $[G,a]$
between $G$ and algebra elements $a$. In Section \ref{sec:nonunital-case} we
will verify these conditions for suitable spectral triples and $G=(1+\D^2)^{-p/2}$.

\subsection{Preliminaries}
Let $\cn$ be a semifinite von Neumann algebra with  
$\tau$, a faithful, normal, semifinite trace. Suppose also that $G$ is a positive and injective operator in $\cn$. We make neither $\tau$-compactness nor summability
hypotheses on $G$ itself. Instead we consider that
a natural condition of integrability for $a\in \cn$, relative to $G$,  would be to ask that  
\begin{equation}
aG^s\in\LL^1\,,\quad\forall s>1\,.
\label{eq:left}
\end{equation}
If we consider $a\geq 0$ then the condition \eqref{eq:left}
is equivalent to
\begin{equation}
G^s a\in\LL^1\,,\quad\forall s>1\,,
\label{eq:right}
\end{equation}
since $\LL^1$ is a $*$-ideal. For $a\geq 0$ satisfying these equivalent conditions,
it is shown by Bikchentaev in  \cite{Bik},
that $a$ also satisfies the equivalent conditions
\begin{equation}
G^{s/2}aG^{s/2}\in\LL^1,\qquad\Longleftrightarrow\qquad
a^{1/2}G^sa^{1/2}\in\LL^1.
\label{eq:sym}
\end{equation}
That is \eqref{eq:left}$\Leftrightarrow$\eqref{eq:right}$\Rightarrow$\eqref{eq:sym}, i.e. integrability of $a\geq 0$ implies square integrability of $a^{1/2}$.
The symmetry of the conditions in \eqref{eq:sym} make them technically much easier 
to work with, however we point out the following.
\begin{lemma}\label{lem:inequiv} 
The two equivalent conditions in \eqref{eq:sym} do not imply the two
equivalent conditions \eqref{eq:left} and \eqref{eq:right}.
\end{lemma}
\begin{proof} We provide a counterexample.
Indeed, choose $\cd\geq1$ affiliated with $\cn$ and pick
$T\in\cn$, a projection, such that $T\cd^{-1}T\in\LL^1$ but
with $\sqrt{T\cd^{-1}T}\notin\LL^1$. Now, define
$\tilde\cd$ affiliated with $M_2(\cn)$ and $\tilde T\in M_2(\cn)$, both acting on
$\tilde \cH=\cH\oplus\cH$, by
$$
\tilde\cd:=
\begin{pmatrix}
1&\sqrt{\cd-1}\\
\sqrt{\cd-1}&-1
\end{pmatrix}
\,,\quad
\tilde T:=
\begin{pmatrix}
T&0\\
0&0
\end{pmatrix}.
$$
Of course, $\tilde \cd$ is invertible with
$$
\quad\tilde\cd^{-1}=
\begin{pmatrix}
\cd^{-1}&\sqrt{\cd^{-1}-\cd^{-2}}\\
\sqrt{\cd^{-1}-\cd^{-2}}&-\cd^{-1}
\end{pmatrix}.
$$
We see at once that
$$
\tilde T\,\tilde \cd^{-1}\,\tilde T=
\begin{pmatrix}
T\cd^{-1}T&0\\
0&0
\end{pmatrix}=\big|\tilde \cd^{-1}\,\tilde T\big|^2,
$$
so that only $\tilde T\,\tilde \cd^{-1}\,\tilde T$ belongs to 
$ \LL^1\big(M_2(\cn),\tau\otimes\tr\big)$,
not $\tilde \cd^{-1}\,\tilde T$ (nor $\tilde T\,\tilde \cd^{-1})$. 
\end{proof}
Earlier examples of this sort and additional information can be found in 
\cite{Bik2} and references therein.
Replacing $\LL^1$ with $\Z_1$ we obtain a counterexample for the
corresponding statement in $\Z_1$. 
Lemma \ref{lem:inequiv}  shows that when we formulate or obtain conditions 
in either of the symmetric forms in
\eqref{eq:sym}, {\em we can not get back to the unsymmetric forms in} 
\eqref{eq:left} and 
\eqref{eq:right}. By a result of one of us, \cite[Lemma 5.10]{LPS}, similar results 
hold for a wide range of operator ideals, in particular for $a,\,G\in\cn$ both 
positive we have
$$
aG\in \Z_1\Leftrightarrow Ga\in\Z_1\ \ \stackrel{\Longrightarrow}{\not\Leftarrow}
\ \ a^{1/2}Ga^{1/2}\in \Z_1
\Leftrightarrow G^{1/2}aG^{1/2}\in \Z_1.
$$
So again, whilst it is technically easier to work with the symmetrised conditions, we can
not recover the unsymmetrised conditions from them. 
When we consider functionals on $\cn$ of the form 
$$
\cn_+\ni a\mapsto\tau(G^{s/2}aG^{s/2}),\ \ \ \ s>1\,,
$$
we must confront the fact that we are dealing with a weight. When we determine the 
domain, it will not be an ideal, but rather a Hilbert algebra. We will not, however, 
go through the details of proving that our algebras satisfy the Hilbert algebra 
axioms as we do not need them here
and will only work with a Banach algebra completion. 
Nevertheless we make use of
some Hilbert algebra ideas and this explains in part the 
`square summable' flavor of some of our hypotheses and results.

We also observe that the symmetric assumption `$a^*G^s a\in \L^1$ for all $s>1$',
equivalent to  `$G^s a\in \L^2$ for all $s>1/2$', is far easier to verify than the 
unsymmetrised condition   `$G^s a\in \L^1$ for all $s>1$'. Indeed,  
an $\L^2$-norm is typically easier to compute than 
an $\L^1$-norm. 

Moreover, in the type I case, that is for $\cn=\B(\cH)$ 
equipped with the operator trace, there is a useful way to check  if the 
product $G^s a$ is  Hilbert-Schmidt. This is a simple corollary of the `little 
Grothendieck Theorem', the Grothendieck factorization principle (see \cite{Defant}). 
The latter states that an operator is Hilbert-Schmidt if and only if it factors through a 
Banach space of the form $L^1(X,\mu)$, for 
$(X,\mu)$ a $\sigma$-finite measure space, or of the form $C_0(K)$, for $K$ a 
locally compact topological space. 
In particular, this applies for a large class of locally compact commutative
spaces.   Let $M$ be a $p$-dimensional complete Riemannian 
manifold and $\Delta$ the associated scalar Laplace-Beltrami operator.
Then, under mild assumptions of the behavior of the geometry at infinity
(for instance those given in \cite{GIV}),
 one readily deduces that the operator 
$M_f(1+\Delta)^{-ps/4}$, is Hilbert-Schmidt whenever $f\in L^2$ and $s>1$. 
Indeed, $M_f$ is obviously continuous from $L^\infty(M)$ (which 
is of the form $C_0(K)$ by the Gelfand-Na{\u\i}mark duality)  to 
$L^2(M)$, while $(1+\Delta)^{-ps/4}$ continuously maps $L^2(M)$ to 
$L^\infty(M)$. This can be shown by standard heat-kernel methods
(see \cite[Lemma 4.5]{GIV}).

\subsection{Banach Algebras from  zeta-function and heat-kernel}
\begin{definition} Given $G$, a positive and injective element of $\cn$,
define the two families  of (possibly infinite)  bilinear functionals on $\cn$
$$
\zeta(a,b;s):=\tau\big(aG^sb\big)\,,\quad g(a,b;\lambda):=\frac1\lambda
\tau\big(ae^{-\lambda^{-1}G^{-1}}b\big)\,,\quad a,b\in\cn\,.
$$
Then, introduce the  following semi-norms on $\cn$
\begin{align*}
\|a\|_\zeta:=\sup_{1\leq s\leq 2}\,\sqrt{s-1}\zeta(a^*,a;s)^{1/2}\,, \quad
\|a\|_{\HK}:=\big\|M g(a^*,a;.)\big\|_\infty^{1/2}\,,
\end{align*}
where $M$ denotes the C\'esaro mean of the multiplicative group $\mathbb R_+^*$.
\end{definition}
Note that since $\langle a,b\rangle _{\zeta,s}:=\zeta(a^*,b,s)$ and
$\langle a,b\rangle_{\HK,\lambda}:=g(a^*,b;\lambda)$ are inner products,
it follows at once that $\|.\|_\zeta$ and $\|.\|_{\HK}$ are positively 
homogeneous and satisfy the triangle
inequality.  But they are also injective maps. Indeed, if for instance, 
$\|a\|_\zeta=0$, then by the faithfulness of $\tau$, we deduce that 
$0=a^*G^sa=|G^{s/2}a|^2$ and thus  $G^{s/2}a=0$ too. Then, from 
the injectivity of $G$, we get that $a=0$. A similar result holds for 
$\|a\|_{\HK}$. This shows that 
$\|.\|_\zeta$ and $\|.\|_{\HK}$ are true norms, not only semi-norms.

By the functional calculus, we see that if $a^*G^sa$ is trace class  for 
some $s>1$, so is $a^*e^{-tG^{-1}}a$ for all $t>0$. Indeed, we have the operator inequality
\begin{equation}
\label{TRE}
a^*e^{-tG^{-1}}a\leq \|G^{-s}e^{-tG^{-1}}\| \,a^*G^sa\leq(s/t)^se^{-s}\,a^*G^sa.
\end{equation}
The
finiteness of such semi-norms is closely related to the zeta function and heat
kernel characterizations of the ideal $\Z_p$ in the unital
case, \cite{CRSS}.
\begin{prop}
\label{B}
Let $B_\zeta(G)$ (respectively $B_{\HK}(G)$) be the normed subset of $\cn$  
relative  to the norm
$\|a\|_\zeta+\|a^*\|_\zeta+\|a\|$ (respectively  $\|a\|_{\HK}+\|a^*\|_{\HK}+\|a\|$). 
Then,  $B_\zeta(G)$ and $B_{\HK}(G)$ are Banach $*$-algebras.
\end{prop}
When no confusion can occur, we write $B_\zeta$ and $B_{\HK}$ instead of
$B_\zeta(G)$ and $B_{\HK}(G)$.
\begin{proof}[Proof of Proposition \ref{B}]
It is clear that $B_\zeta$ and $B_{\HK}$ are normed linear spaces with 
symmetric norms.
Let us first show that they are sub-multiplicative. For a positive numerical 
function $f$, we have:
$$
\tau\big((ab)^*f(G)ab\big)\leq\|b\|^2\|a^*f(G)a\|_1=
\|b\|^2\tau\big(a^*f(G)a\big)\,.
$$
This clearly entails that
$\|ab\|_\zeta\leq\|b\|\|a\|_\zeta$ and $ \|(ab)^*\|_\zeta\leq\|a\|\|b^*\|_\zeta$,
and thus
\begin{align*}
\|ab\|_\zeta+\|(ab)^*\|_\zeta+\|ab\|&\leq \|b\|\|a\|_\zeta+\|a\|\|b^*\|_\zeta+\|a\|\|b\|\\
&\leq\big(\|a\|_\zeta+\|a^*\|_\zeta+\|a\|\big)\big(\|b\|_\zeta+\|b^*\|_\zeta+\|b\|\big)\,.
\end{align*}
A similar statement holds for the norm of $B_{\HK}$.

It remains to prove that $B_\zeta$ and $B_{\HK}$ are complete for the norm topology
we have introduced.
But this is an easy exercise since a Cauchy sequence $(a_n)_{n\in\mathbb N}$ in $B_\zeta$  is convergent as an element of $\cn$ (since the norm of $B_\zeta$ dominates the one of $\cn$). Call $a\in\cn$ such a limit and suppose that $\|a\|_\zeta=\infty$, i.e. $\lim \|a_n\|_\zeta=\infty$. This is a contradiction with the assumption of $(a_n)_{n\in\mathbb N}$ being $B_\zeta$-Cauchy since by the second triangle inequality we have $\big|\|a_n\|_\zeta-\|a_m\|_\zeta\big|\leq \|a_n-a_m\|_\zeta$.
The arguments for $B_{\HK}$ are similar.
\end{proof}
The algebras $B_\zeta,\,B_{\HK}$ need not  be uniformly closed, nor
weakly closed, nor be ideals (even one-sided) in $\cn$.
We now prove that these two notions of `square integrability'  in fact coincide.
\begin{theorem}
\label{NE}
The norms $\|.\|_\zeta$  and $\|.\|_{\HK}$ are equivalent.
\end{theorem}
\begin{proof}
 Fix any $a\in\cn$ with $\|a\|_\zeta<\infty$. We
need to show that the associated function $g(a^*,a;.)$ has bounded Cesaro
mean.
\begin{align*}
\big(Mg(a^*,a;.)\big)(x)&=\frac1{\log x}\int_1^x\,g(a^*,a;\lambda)\,
\frac{d\lambda}\lambda
=\frac1{\log x}\int_1^x\,\frac1\lambda
\tau\big(a^*e^{-\lambda^{-1}G^{-1}}a\big)\,
\frac{d\lambda}\lambda\\
&=\frac1{\log x}\int_{x^{-1}}^1\,
\tau\big(a^*e^{-tG^{-1}}a\big)\,dt
=\frac1{\log x}\tau\big(a^*G
e^{-x^{-1}G^{-1}}a\big)
-\frac1{\log x}\tau\big(a^*
Ge^{-G^{-1}}a\big)\\
&\leq\frac1{\log x}\tau\big(a^*G
e^{-x^{-1}G^{-1}}a\big).
\end{align*}
The inversion of the trace and the integral can be justified by
Fubini's Theorem. Indeed,  the operator-valued function
$
t\mapsto \chi_{[x^{-1},1]}(t)\,a^*e^{-tG^{-1}}a\,,
$ belongs to $C([x^{-1},1],\L^1(\cn,\tau))$ and thus to 
$L^1\cap L^\infty([x^{-1},1],\L^1(\cn,\tau))$. Using the embedding of
$C([x^{-1},1],\cn)$ into the von Neumann algebra $L^\infty([x^{-1},1],dt)\otimes\cn$,
we see that it defines  an integrable element of 
$L^\infty([0,1],dt)\otimes\cn$ for the faithful semifinite trace
$\int\otimes\tau$.
Now, making the change of variable $x=e^{1/\eps}$ (i.e. $\eps=s-1$)
in the previous expression, we obtain, for $0<\eps<1$,
\begin{align*}
\big|\big(Mg(a^*,a;.)\big)(e^{1/\eps})\big|
\leq
\eps\tau\big(a^*G
e^{-e^{-1/\eps}G^{-1}}a\big)
\leq
\eps\tau\big(a^*G^{1+\eps}a\big)
\big\|G^{-\eps} e^{-e^{-1/\eps}G^{-1}}\big\|\,,
\end{align*}
where we have used the operator inequality
$$
a^*Ge^{-e^{-1/\eps}G^{-1}}a\leq
\big\|G^{-\eps} e^{-e^{-1/\eps}G^{-1}}\big\|\,
a^*G^{1+\eps}a\,.
$$
But the function $x\mapsto x^\eps e^{-e^{-1/\eps}x}$ reaches its
maximum at $x=\eps e^{1/\eps}$, where its value is $e
\eps^\eps e^{-\eps}\leq e$. Thus,
\begin{align*}
\big|\big(Mg(a^*,a;.)\big)(e^{1/\eps})\big|
\leq \,e\,\eps\,\tau\big(a^*G^{1+\eps}a\big)=e\,\eps\,\zeta(a^*,a;1+\eps)\,,
\end{align*}
which concludes the proof of the first inclusion, since the right
hand side is bounded by assumption.

To show that $\|.\|_\zeta\leq C\|.\|_{\HK}$, we will use the Laplace transform.
Fix any $a\in \cn$ with finite  ${\HK}$-norm. We have, writing again $\eps=s-1$,
$$
G^{1+\eps}=\frac1{\Gamma(1+\eps)}\int_0^\infty t^\eps
\,  e^{-tG^{-1}} \, dt.
$$
Disregarding the bounded pre-factor $\Gamma(1+\eps)^{-1}$ in the
expression above, 
we first decompose the integral into two pieces:
$\int_0^\infty=\int_0^1+\int_1^\infty$, to write
$a^*G^{1+\eps}a$ as a sum of two operators.
For the second term, we obtain
\begin{align*}
\tau\Big(\int_1^\infty
t^\eps\,a^*  e^{-tG^{-1}} a\,dt\Big)
&\leq
\int_1^\infty
t^\eps\,  \tau\big(a^* \,e^{-G^{-1}}a\big) \,
\big\|e^{-(t-1)G^{-1}}\big\|\,dt\\
&\hspace{5cm}=
 g(a^*,a;1) \,
\int_1^\infty
t^\eps\,e^{-(t-1)\|G\|^{-1}}\,dt\leq C_1,
\end{align*}
where the constant $C_1$ is independent of $\eps$.
For the first term, we can exchange the trace and the integral because of  the
finite range of the integration. This reads,
\begin{align*}
\tau\Big(\int_0^1
t^\eps\,  a^*e^{-tG^{-1}}a \,dt\Big)&=
\int_0^1
t^\eps\,  \tau\big(a^*e^{-tG}a\big) \,dt
=\int_1^\infty
\lambda^{-\eps}\,g(a^*,a;\lambda)\,\frac{d\lambda}{\lambda}.
\end{align*}
Now, let us again decompose the integral as
$\int_1^\infty=\int_1^{e^{\eps^{-1}}}+\int_{e^{\eps^{-1}}}^\infty$.
For the first part, we can conclude that
\begin{align*}
\int_1^{e^{\eps^{-1}}}\!\!\!
\lambda^{-\eps}g(a^*,a;\lambda)\,\frac{d\lambda}\lambda&\leq
\int_1^{e^{\eps^{-1}}}\!\!\!
g(a^*,a;\lambda)\,\frac{d\lambda}{\lambda}=\,\frac1\eps\big(Mg(a^*,a;.)\big)
(e^{\eps^{-1}})\leq\frac{\|Mg(a^*,a;.)\|_\infty}{\eps}=:\frac{C_2}\eps.
\end{align*}
The last term is
\begin{align}
\label{LAST}
\int_{e^{\eps^{-1}}}^\infty
\lambda^{-\eps}\,g(a^*,a;\lambda)\,\frac{d\lambda}{\lambda}.
\end{align}
Now, we are going to make use of the following
change of variable:
$0\leq y(\lambda):=\int_1^\lambda g(a^*,a;\sigma)\,
\frac{d\sigma}{\sigma}$,
 a monotonically increasing function of $\lambda$. Observing that
$y(\lambda)=\big(Mg(a^*,a;.)\big)(\lambda)\log(\lambda)$, there exists a 
positive constant $C_3$, such that 
$y(\lambda)\leq C_3\log(\lambda)$, and thus $\lambda^{-\eps}\leq e^{-\eps C_3^{-1}y}$. 
This implies that the expression \eqref{LAST} is smaller than
$$
\int_0^\infty
e^{-\eps C_3^{-1}y}\,dy=\eps^{-1}\int_0^\infty
e^{-C_3^{-1}y}\,dy=:\frac{C_3}{\eps}.
$$
Gathering these estimates together  proves that 
$$
\eps\,\tau\big(a^*G^{1+\eps}a\big)\leq \eps\Big(C_1+\frac{C_2}\eps+\frac{C_3}\eps\Big),
$$
and thus the set 
$\big\{(s-1)\zeta\big(a^*,a;s\big):1\leq s\leq2\big\}$ is
bounded and so $\|a\|_\zeta<\infty$. This concludes the proof of the
proposition.
\end{proof}

Applying this result to the unital case, i.e. when $G$ alone belongs to $\Z_1$, and 
combining it with Theorem \ref{great}, we get an interesting fact.
\begin{corollary} The following three norms on $\Z_1$ are equivalent.
$$
\|T\|_{1,\infty}=\sup_{t>0}\frac1{\log(1+t)}\int_0^t\mu_s(T)ds,\quad 
\sup_{s>1} (s-1)\tau(|T|^s),\quad
\sup_{\lambda>0}\frac1{\log(\lambda)}\int_1^\lambda \mu^{-2}
\tau\big(e^{-\mu^{-1}|T|^{-1}}\big)d\mu.
$$
\end{corollary}

A positive element $\omega\in L^\infty(\mathbb R)^*_+$, is called a Banach limit, 
if it is translation invariant and satisfies 
${\rm ess}-\liminf _{t\to \infty}(f)\leq\omega(f)\leq {\rm ess}-\limsup _{t\to \infty}(f)$, 
for all $0\leq f\in  L^\infty(\mathbb R)$.  In particular, we see that any Banach 
limit must vanish on $C_0(\mathbb R)$.
Given such a Banach limit $\omega$, we define the dilation invariant
functional $\tilde\omega$ on 
$\R^*_+$, by $\tilde\omega:=\omega\circ\log$. 
We use the notation $\omega-\lim_{t\to\infty} f(t)$ instead 
of $\omega(f)$. Next we show that the 
$\tilde\omega$-residue of the zeta function $\zeta(a^*,a;.)$ 
coincides  with the $\omega$-limit  of the C\'esaro mean 
of the heat-trace function $g(a^*,a;.)$.

\begin{theorem}
\label{OL}
Let $a\in B_\zeta$, then for any Banach limit  $\omega$, we have
$$
\omega-\lim_{\lambda\to\infty} M\big(g(a^*,a;.)\big)(\lambda)
=\tilde\omega-\lim_{r\to\infty}\frac1r\,\zeta\big(a^*,a;1+\tfrac1r\big).
$$
If moreover $\omega$ is $M$-invariant and  $g(a,a;.)$ is bounded, then
$$
\omega-\lim_{\lambda\to\infty} g(a^*,a;\lambda)=\tilde\omega-\lim_{r\to\infty}\frac1r\,\zeta\big(a^*,a;1+\tfrac1r\big).
$$
In both  these situations, if one of the ordinary limits exists, both limits exist, 
and they coincide.
\end{theorem}
\begin{proof}
We only prove the first part, the remaining statements  will then follow 
immediately from general properties of Banach limits.

In the course of  the proof of Theorem \ref{NE}, we have shown that
\begin{align*}
\Big\|a^*G^{1+\frac1r}a-\frac1{\Gamma(1+\frac1r)}\int_0^1t^{\frac1r}
a^*e^{-tG^{-1}}a\,dt\Big\|_1&\leq g(a^*,a;1)\int_1^\infty t^{\frac1r}e^{-(t-1)\|G\|^{-1}}dt.
\end{align*}
Hence
$$
\lim_{r\to\infty}\frac1r\Big[\tau\big(a^*G^{1+\frac1r}a\big)-\frac1{\Gamma(1+\frac1r)}\int_0^1t^{\frac1r}
\tau\big(a^*e^{-tG^{-1}}a\big)\,dt\Big]=0,
$$
thus setting $\tilde\omega=\omega\circ\log$, we have:
\begin{align*}
\tilde\omega-\lim_{r\to\infty}\frac1r \tau\big(a^*G^{1+\frac1r}a\big)= 
\tilde\omega-\lim_{r\to\infty}\frac1r
\frac1{\Gamma(1+\frac1r)}\int_0^1t^{\frac1r}
\tau\big(a^*e^{-tG^{-1}}a\big)\,dt.
\end{align*}

Substituting $t=e^{-\mu}$, a little computation shows that
\begin{equation}
\label{WK}
\int_0^1t^{\frac1r}\tau\big(a^*e^{-tG^{-1}}a\big)\,dt
=\int_0^\infty e^{-\frac \mu r} \,d\beta(\mu)\,\quad \mbox{where}\quad \beta(\mu)
=\int_0^\mu e^{-\nu}\tau\big(a^*e^{-e^{-\nu}G^{-1}}a\big)\,d\nu.
\end{equation}
We are now wish to use the weak-$*$ Karamata Theorem \cite[Theorem 2.2]{CPS2}, and
need to check the various hypotheses. 
First, $\omega$ is a translation invariant mean on $\mathbb R$, so $\tilde\omega$ 
is a dilation invariant mean on $\mathbb R_+^*$. 
Next, $\beta$ is positive, increasing and  continuous on $\R^+$, and satisfies 
$\beta(0)=0$. Finally, we need to check that $\int_0^\infty e^{-\frac \mu r}d\beta(\mu)$ is 
finite for any $r>0$.
But this follows immediately from the first equality in Equation \eqref{WK}. Hence, the 
weak-$*$ Karamata Theorem gives us
$$
\tilde\omega-\lim_{r\to\infty}\frac1r\int_0^\infty e^{-\frac \mu r} \,d\beta(\mu)
=\tilde\omega-\lim_{\mu\to\infty}\frac{\beta(\mu)}\mu.
$$
But 
$$
\frac{\beta(\mu)}\mu
=\frac1\mu\int_1^{e^\mu} \lambda^{-2}\tau\big(a^*e^{-\lambda^{-1}G^{-1}}a\big)\,d\lambda
=\big(M g(a^*,a;.)\big)(e^\mu)\,,
$$
from which the result follows, since
$$
\tilde\omega-\lim_{r\to\infty}\frac1r\tau\big(aG^{1+\frac1r}a\big)=
\tilde\omega-\lim_{\mu\to\infty}\big(M g(a,a;.)\big)\circ\exp(\mu)
=\omega-\lim_{\mu\to\infty}\big(M g(a,a;\mu)\big).
\hspace{1cm}\qed
$$
\hideqed
\end{proof}

 To conclude this discussion, we give  some obvious but useful stability properties of $B_\zeta$.

\begin{lemma}
\label{Bp-grew}
i) If $a\in B_\zeta$ and $f\in L^\infty(\R)$, then $af(G)\in B_\zeta$, and 
when $a^*=a$, $af(a)\in B_\zeta$.\\
ii) If $0\leq a,b\in\cn$ are such that $a \in B_\zeta$ and $b^2\leq a^2$, then $b\in B_\zeta$.\\
iii) If $a\in B_\zeta$, then $|a|,|a^*|\in B_\zeta$.
\end{lemma}
\begin{proof}
From the ideal property  of the trace-norm we obtain
$$
\tau\big(\bar f(G) a^*G^sa f(G)\big)\leq \|f\|_\infty^2 
\tau\big(a^*G^sa\big)\,,
$$
and from the operator inequality
$$
a f(G) G^s\bar f(G)a^*\leq \|f\|_\infty^2 aG^sa^*\,,
$$
we obtain the first part of $i)$. The second part of $i)$ is even more immediate. 

To prove $ii)$, note that from \eqref{eq:sym}, that is from the trace property, we have
\begin{align*}
\tau\big(bG^sb\big)&=\tau\big(G^{s/2}b^2G^{s/2}\big)\leq
\tau\big(G^{s/2}a^2G^{s/2}\big)=\tau\big(aG^sa\big)\,.
\end{align*}

Finally, to obtain $iii)$, let $u|a|$ and $v|a^*|$ be the polar decomposition of $a$ and $a^*$. Then of course, $|a|=u^*a=a^*u$ and $|a^*|=v^*a^*=av$. Thus,
$$
\tau\big(|a| G^s|a|\big)=\tau\big(u^*a G^sa^*u\big)\leq \tau\big(a G^sa^*\big)\,.
$$
The proof for $|a^*|$ is entirely similar.
\end{proof}

\subsection{Relations between $B_\zeta(G)$, $\Z_p$ and  $\L^{p,\infty}$}

\begin{prop}
\label{thm:Z_1-landing}
If $a,b\in B_\zeta(G)$, then
$bGa\in\cz_1$. Moreover, when $b=a^*$, we have the following partial-trace estimate:
\begin{align*}
\sigma_t\big(a^*Ga\big)&:=\int_0^t \mu_s\big(a^*Ga\big)\,ds
\leq \frac1e\,\big\|a^*\,e^{-G^{-1}}\,a\big\|_1+\|a\|^2+\|Mg(a^*,a;.)\|_\infty \log(1+t)\,.
\end{align*}
\end{prop}

\begin{proof}
By \cite[Lemma 2.3]{ChS}  we have the  inequality:
$$
\|bGa\|_{1,\infty}\leq\||bG^{1/2}|^2\|_{1,\infty}^{1/2}\||G^{1/2}a|^2\|_{1,\infty}^{1/2}=
\||G^{1/2}b^*|^2\|_{1,\infty}^{1/2}\||G^{1/2}a|^2\|_{1,\infty}^{1/2}=
\|bGb^*\|_{1,\infty}^{1/2}\|a^*Ga\|_{1,\infty}^{1/2},
$$
so we may assume without loss of generality 
that $b=a^*$.

Recall from \cite[Lemma 3.3]{CPS2} that 
if $ a\in\cn$ has norm $||a||\leq M$, then for any $1\leq s <2$,
$$
(a^*Ga)^s\leq M^{2(s-1)}a^*G^sa.
$$
Using this inequality we have
\begin{align*}
 \limsup_{s\to 1}(s-1)\tau((a^*Ga)^s)\leq \limsup_{s\to 1}\Vert a\Vert^{2(1-s)}(s-1)
\tau(a^*G^sa)<\infty\,,
\end{align*}
which shows that $a\in B_\zeta\Rightarrow  a^*Ga\in\Z_1$.

To obtain the estimate of the partial trace, we use the Laplace transform to write
\begin{align*}
a^*Ga
=\int_0^\infty
a^*\,e^{-tG^{-1}}\,a\,dt
&=\Big(\int_0^{e^{-k}}+\int_{e^{-k}}^1+\int_1^\infty
\Big)
a^*\,e^{-tG^{-1}}\,a\,dt
=:C_k+B_k+A.
\end{align*}
By \eqref{TRE},  we see that $A$ is trace-class with
$\|A\|_1\leq e^{-1}\big\|a^*\,e^{-G^{-1}}\,a\big\|_1$
and we focus on the rest. For $C_k$, we
have the bound
$
\|C_k\|\leq\|a\|^2\,e^{-k}.
$
The operator $B_k$ can be  bounded in trace-norm:
\begin{align*}
\|B_k\|_1\leq\int_{e^{-k}}^1
\,\tau \big(a^*\,e^{-tG^{-1}}\,a\big)\,dt
=\int_1^{e^k}g(a^*,a;\lambda)\,
\frac{d\lambda}\lambda&=\ln(e^k)\,\big(Mg(a^*,a;.)\big)(e^k)\\
&\hspace{3,5cm}\leq \|Mg(a^*,a;.)\|_\infty\,k\,.
\end{align*}
Recall the $K$-functional associated to the Banach couple $\big(\LL^1,\cn\big)$
$$
K(T,t;\LL^1,\cn):=\inf\big\{\|T_1\|_1+t\|T_2\|,\,T=T_1+T_2,\,T_1\in\L^1,\,T_2\in\cn\big\}.
$$
It is known that in this case, it can be exactly evaluated and reads
$$
K(T,t;\LL^1,\cn)=\int_0^t \mu_s(T)ds=:\sigma_t(T).
$$
Thus, writing $D=C_k+B_k$, we can estimate
$$
\sigma_s(D)\leq \|B_k\|_1+s\|C_k\|
\leq\|Mg(a^*,a;.)\|_\infty\,k+\|a\|^2 se^{-k}.
$$
Finally, given  $s\in(0,\infty)$, define $k\in\real_+$ as
$
k=\ln(1+s),
$
so that we have
$$
\sigma_s(D)
\leq\|Mg(a^*,a;.)\|_\infty\,\ln(1+s)+
\|a\|^2s(1+s)^{-1}.
$$
Gathering these estimates together, we find the bound stated in the lemma.
\end{proof}

The next result refines our approach to obtain containments in $\Z_{q}$, $q\geq 1$.

\begin{prop}
\label{complex-interpol}
Let  $\delta \in(0,1]$ and $0\leq a\in \cn$ be such that $aGa\in\Z_1$. Then, for any $\eps\in(0,\delta/2]$, $a^\delta G^\eps\in\Z_{1/\eps}$ with
$$
\|a^\delta G^\eps\|_{1/\eps,\infty}\leq \|a\|^{\delta-2\eps} \|aGa\|_{1,\infty}^\eps \,.
$$
\end{prop}
From Theorem \ref{thm:Z_1-landing}, we see that the 
assumption  $aGa\in\Z_1$ is satisfed for $a\in B_\zeta$. 

\begin{proof}[Proof of Proposition \ref{complex-interpol}]
Note that the statement is equivalent to:
$$
aGa\in\Z_1\Rightarrow a^\delta G^\eps a^\delta\in\Z_{1/\eps},\quad \forall \eps\in(0,\delta]\,.
$$
Consider the holomorphic operator valued function on the open  strip 
$S=\{z\in\mathbb C:\Re z\in(0,1)\}$, given by $F(z)= a^zG^za^z$.
For all $y\in\R$, we have $F(iy)\in\cn$ with $\|F(iy)\|\leq 1$. Moreover, 
$F(1+iy)\in \Z_1$. Indeed, since 
$$
\int_0^t\mu_s(AB)ds\leq\int_0^t\mu_s(A)\mu_s(B)ds\,,\quad \forall A,B\in\cn\,,
$$
proven in \cite[Theorem 4.2, iii)]{FK} and \cite[Proposition 1.1]{ChS}, we obtain
\begin{align*}
\int_0^t\mu_s(F(1+iy))ds&=\int_0^t\mu_s(a^{1+iy}G^{1+iy}a^{1+iy})ds\leq
\int_0^t\mu_s(aG^{1+iy}a)ds\\
&\hspace{-1cm}\leq\int_0^t\mu_s(aG^{1/2})\mu_s(G^{1/2+iy}a)ds\leq
\int_0^t\mu_s(aG^{1/2})\mu_s(G^{1/2}a)ds=\int_0^t\mu_s(aGa)ds\,,
\end{align*}
and thus 
$\|F(1+iy)\|_{1,\infty}\leq \|aGa\|_{1,\infty}$.

This shows that  $a^\eps G^\eps a^\eps$ belongs to the first complex interpolation 
space $(\Z_1,\cn)_{[\eps]}$ and hence belongs to $(\Z_1,\cn)^{[\eps]}$, the 
second complex interpolation space. (There is a norm one  injection from 
$(X_0,X_1)_{[\theta]}$ into $(X_0,X_1)^{[\theta]}$, for any Banch couple 
$(X_0,X_1)$, $\theta\in[0,1]$). But the latter is  $\Z_{1/\eps}$, as shown in  
subsection \ref{interpolation}. In summary,  we have
$$
\|a^\eps G^\eps a^\eps\|_{1/\eps,\infty}=\|F(\eps)\|_{(\Z_1,\cn)^{[1-\eps]}}\leq \|F(\eps)\|_{(\Z_1,\cn)_{[1-\eps]}}\leq \|F(0)\|^{1-\eps}\|F(1)\|_{1,\infty}^\eps= \|aGa\|_{1,\infty}^\eps\,,
$$
and from the ideal property, we obtain the announced result
$$
\|a^\delta G^\eps a^\delta\|_{1/\eps,\infty}\leq \|a\|^{2(\delta-\eps)}\|aGa\|_{1,\infty}^\eps\,,\quad \forall \delta\in(0,1]\,,\quad \forall \eps\in(0,\delta]\,.\hspace{2cm}\qed
$$
\hideqed
\end{proof}

According to Lemma \ref{lem:inequiv}, the assumption that
$a,b\in B_\zeta$ is not enough to ensure that
$abG$ belongs to $\Z_1$.
On a more positive note, the intuitive result that when $g(a^*,a;.)$ is already bounded, that
$a^*G^\eps a$, $\eps\in(0,1)$, is in the small ideal $\cl^{1/\eps,\infty}$ is true.

\begin{prop}\label{prop:lille}
Let $\eps\in(0,1)$ and let $a\in\cn$ be such that the map 
$\mathbb R^+\ni\lambda\mapsto g(a^*,a;\lambda)$ is bounded. Then 
$a^*G^\eps a\in\cl^{1/\eps,\infty}$, with
\begin{equation}
\label{sing-nat}
\mu_s\big(a^*G^\eps a\big)\leq \frac1{\Gamma(\eps)}\Big(\frac1\eps \|a\|^2+\frac 1{1-\eps}\|g(a^*,a;.)\|_\infty \Big)s^{-\eps}\,.
\end{equation}
\end{prop}
\begin{proof}
We write
$$
a^*G^\eps a=
\frac1{\Gamma(\eps)}\int_0^\infty t^{\eps-1}\,a^*e^{tG^{-1}}a\,dt\,,
$$
and split $\int_0^\infty=\int_{e^{-k}}^\infty+\int_0^{e^{-k}}$,
to obtain $a^*G^\eps a=B_k+C_k,\quad k\in\real^*_+$.
We  notice that for any $S\in\cn$,
$$
\|S\|_1=\int_0^\infty\mu_t(S)\,dt\geq\int_0^s\mu_t(S)\,dt\geq s \,\mu_s(S)\,.
$$
Then, using Fan's inequality, we obtain
$$
\mu_s(B_k+C_k)\leq\mu_0(C_k)+\mu_s(B_k)\leq\|C_k\|+s^{-1}\|B_k\|_1.
$$
We have first $\|C_k\|\leq C\|a\|^2 e^{-\eps k}$. Indeed
$$
\|C_k\|\leq\frac1{\Gamma(\eps)}\|a\|^2\int_0^{e^{-k}}  t^{\eps-1}dt=
\frac1{\eps\Gamma(\eps)}\|a\|^2e^{-\eps k}\,.
$$

For the second part, we have
\begin{align*}
\|B_k\|_1\leq\frac1{\Gamma(\eps)}\int_{e^{-k}}^\infty
\,\tau \big(a^*\,e^{-tG^{-1}}\,a\big)\,t^{\eps-1}\,dt
&=\frac1{\Gamma(\eps)}\int_{e^{-k}}^\infty\,g\big(a^*,a;t^{-1}\big)\,
t^{\eps-2}\,dt\\
& \leq \frac1{\Gamma(\eps)}\|g(a^*,a;.)\|_\infty (1-\eps)^{-1}e^{k(1-\eps)}\,.
\end{align*}
Thus we have
$$
\mu_s(B_k+C_k)\leq\frac1{\Gamma(\eps)}\Big(\frac1\eps \|a\|^2+\frac 1{1-\eps}\|g(T^*,T;.)\|_\infty \frac{e^{k}}s\Big)e^{-\eps k}\,.
$$
So, if for each $s\in\real^+$, we choose  $k=\log s$, we obtain the desired estimate.
\end{proof}

{\bf Remark}. By polarization,  the positivity assumption
in Proposition \ref{prop:lille} may be eliminated.
That is, for $a,b\in B_\zeta$, both having bounded $g$-function, we have
$aG^\eps b\in\cl^{1/\eps,\infty}$, for all $\eps\in(0,1)$.  
Note finally the singular nature of the estimate 
in equation \eqref{sing-nat} in the limit $\eps \to 1^-$. This prevents us 
reaching the weak-$\L^1$-space, $\L_{1,w}$ of definition \ref{weak-L1}, with these techniques.

\section{Zeta functions and Dixmier traces}\label{sec:zeta-and-dixy}

This Section contains the main application of our previous results. We are interested
in the question, first raised in \cite[Chapter 4]{Co1}, and
further studied in considerable detail in 
\cite{CPS2, CS,
CRSS, LSS,  LS, LPS,Sed} in the unital case, concerning
the relationship between singularities of the zeta function and the Dixmier trace.
The extension of this result to the nonunital case without appealing to the existence of
local units has interested a number of authors. The construction of our Banach algebras
 $B_\zeta$ was motivated by this question.
 The next Section collects some general lemmas needed later.

\subsection{General facts}
We first prove a result which allows us to manipulate the commutator of fractional powers. We are 
indebted to Alain Connes for communicating the proof to us, which we reproduce here for completeness.

\begin{lemma}
\label{commutators-powers}
Let $0\leq A,B\in\cn$ be such that $[A,B]\in\mathfrak S$, where $\mathfrak S$ denotes any symmetrically normed (or quasi-normed) ideal of $\cn$. Denoting by $\mathfrak S_p$, $p\geq 1$, the $p$-convexification of $\mathfrak S$, for all $\alpha,\beta\in(0,1]$, we have
$[A^\alpha,B^\beta]\in\mathfrak S_{1/{\alpha\beta}}$, with
$$
\|[A^\alpha,B^\beta]\|_{\mathfrak S_{1/{\alpha\beta}}}\leq \|A\|^{\alpha(1-\beta)}\|B\|^{\beta(1-\alpha)}\|[A,B]\|_{\mathfrak S}^{\alpha\beta}\,.
$$
\end{lemma}

\begin{proof}
By homogeneity, we can assume that $\|A\|=\|B\|=1$.
We are going to use the Cayley transform twice  to obtain a commutator estimate from a difference estimate and then use the BKS inequality \cite{BKS}. To this end, let $U$ be the unitary operator
$U:=(i+B)(i-B)^{-1}$.
A quick computation shows that
$[A,U]=2i(i-B)^{-1}[A,B](i-B)^{-1}$, which gives
$$
U^*AU-A=U^*[A,U]=2i\,U^*(i-B)^{-1}[A,B](i-B)^{-1}\,.
$$
Thus, we see that $\|U^*AU-A\|_{\mathfrak S_p}=\|[A,B]\|_{\mathfrak S_p}$.
Using finally that
$(U^*AU)^\alpha=U^*A^\alpha U$, $\forall \alpha>0$,
and the BKS inequality
$
\|X^\alpha-Y^\alpha\|_{\mathfrak S_{1/\alpha}}\leq\|X-Y\|_{\mathfrak S}^\alpha\,,
$
we obtain
$$
[A,B]\in\mathfrak S\,\,\Leftrightarrow \,\,U^*AU-A \in\mathfrak S\,\,
\Rightarrow \,\,U^*A^\alpha U-A^\alpha \in\mathfrak S_{1/\alpha} \,\,\Leftrightarrow\,\, 
[A^\alpha,B]\in\mathfrak S_{1/\alpha}\,.
$$
One concludes the proof using the same trick with the unitary
$V:=(i+A^\alpha)(i-A^\alpha)^{-1}$.
\end{proof}

\begin{lemma}
\label{not-uniform}
Let $A,B\in\cn$, $B^*=B$,  such that $[A,B]\in\mathfrak S$ for any 
symmetrically normed ideal of $\cn$ and let $\vf\in C_c^\infty(\R)$. 
Then $[A,\vf(B)]\in\mathfrak S$ with
$$
\|[A,\vf(B)]\|_{\mathfrak S}\leq\|\widehat{\vf'}\|_1\,\|[A,B]\|_{\mathfrak S}\,.
$$
\end{lemma}

\begin{proof}
Since $\vf$  is a smooth compactly supported function, it is the Fourier transform of a Schwartz function $ \widehat {\vf}$ and thus
$
\vf(B)=\int_\R \widehat {\vf}(\xi)\,e^{-2i\pi\xi B}\,.
$
The result then follows from the identity
$$
[A,e^{-2i\pi\xi B}]=-2i\pi\xi\int_0^1\,e^{-2i\pi\xi sB}\,[A,B]\,e^{-2i\pi\xi (1-s)B}\,ds\,.\qed
$$
\hideqed
\end{proof}

Stronger estimates than that given above are available from \cite{PS1,PS2}. 
Finally we will need

\begin{lemma}
\label{limits}
i) If $T\in\Z_1^0$, then 
$
\lim_{\eps\to0^+} \eps\,\|T\|_{1+\eps}=0\,.
$\\
ii) Let $a\in B_\zeta$ and $\delta\in(0,1]$. Then the map 
$(0,\delta/2)\ni\eps\mapsto\|a^\delta G^\eps\|_{1+1/\eps}$ is bounded.
\end{lemma}

\begin{proof}
The first claim follows from \cite[Theorem 4.5 i)]{CRSS}. 

To prove the second part, note that for an arbitrary $T\in \Z_1$, by the definition of the norm in the Marcinkiewicz space $\Z_1$, we have $\mu_t(T)\prec\prec
\|T\|_{1,\infty}/(1+t)$. Since the Schatten spaces $\L^p$,
$ 1\le p\le\infty$, are fully symmetric operator spaces we thus have
$$
\|T\|_p\le \|T\|_{1,\infty}\|[t\mapsto(1+t)^{-1}]\|_p,\quad p>1\,,
$$
that is
\begin{equation}
\label{untruc}
\|T\|_p\leq \|T\|_{1,\infty}({p-1})^{-1/p}\,.
\end{equation}
Let $T:=( a^\delta G^{2\eps} a^\delta)^{1/2\eps}$.
This operator belongs to $\Z_1$,  because $aGa\in\Z_1$ by 
Theorem \ref{thm:Z_1-landing} and thus 
$a^\delta G^{2\eps} a^\delta \in\Z_{1/2\eps}$ by Proposition \ref{complex-interpol}.
Applying the estimate \eqref{untruc}, with $p=1+\eps$, to this operator yields
\begin{align}
\label{labell}
\|G^\eps a^\delta\|_{1+1/\eps}=\|( a^\delta G^{2\eps} a^\delta)^{1/2+1/2\eps}\|_1^{\eps/(1+\eps)}&=
\|( a^\delta G^{2\eps} a^\delta)^{1/2\eps}\|_{1+\eps}^\eps\nonumber\\
&\leq \eps^{-\frac\eps{1+\eps}}\| (a^\delta G^{2\eps} a^\delta)^{1/2\eps}\|_{1,\infty}^\eps
=\eps^{-\frac\eps{1+\eps}}\| a^\delta G^{2\eps} a^\delta\|_{1/2\eps,\infty}^{1/2}\,.
\end{align}
But  Proposition \ref{complex-interpol} gives also the inequality
\begin{equation}
\label{label}
\|a^\delta G^{2\eps}a^\delta \|_{1/2\eps,\infty}\leq\|a\|^{2(\delta-2\eps)}\|aGa\|_{1,\infty}^{2\eps}\,,\quad\forall\delta\in(0,1]\,,\quad \forall\eps\in(0,\delta/2)\,.
\end{equation}
Combining \eqref{labell} with \eqref{label}, we get
$$
\|G^\eps a^\delta\|_{1+1/\eps}\leq \eps^{-\frac\eps{1+\eps}}
\| a^\delta G^{2\eps} a^\delta\|_{1/2\eps,\infty}^{1/2}\leq
\eps^{-\frac\eps{1+\eps}}\|a\|^{\delta-2\eps}\|aGa\|_{1,\infty}^{\eps}\,.
$$
This proves the claim since $\eps^{-\frac\eps{1+\eps}} \to 1$.
\end{proof}

\subsection{Approximation schemes }

Without loss of generality we may assume throughout that 
$G^{-1}\geq 1$.
In the following, we fix $0\leq a\in \cn$  and we assume further that there exists 
$\delta>0$ with $a^{1-\delta}\in B_\zeta$.  We stress that while purely technical 
at the first glance, this extra $\delta$-condition turns out to be the key assumption 
to get an equivalence $a\in B_\zeta\Leftrightarrow a^*Ga\in\Z_1$. 
This is explained in  Section \ref{conter-ex}.
We then construct a pair of approximation processes for $a$ in the strong topology, the
first being given by 
$$
a_n:=a P_n,\quad \mbox{where}\quad P_n:=\int_{1/n}^{\|a\|}dE_a(\lambda),
$$
with $\int_0^{\|a\|}\lambda\, dE_a(\lambda)$ the spectral resolution of $a$. 
Note  the operator inequality $a_n^2\geq\frac1{n^2} P_n$. Lemma \ref{Bp-grew} ii) 
implies then that $P_n\in B_\zeta$ as well.

Next, for each $n\in\mathbb N$, we pick  $0\leq\vf_n\in  C^\infty_c(\R)$ such that, 
restricted to the interval $[0,\|a\|]$, we have
$\chi_{(1/n,\|a\|]}\leq \vf_n\leq \chi_{(1/(n+1),\|a\|]}$.
This immediately implies that
$P_n\leq\vf_n(a),\,\vf_n^2(a)\leq P_{n+1}$.
For the second limiting process we define  $a^\vf_n:=a\vf_n(a)$.
Now, since
$$
a^\delta \,(1-\vf_n(a))\leq
a^\delta \,(1-\vf_n^2(a))\leq a^\delta (1-P_n)=\int_0^{1/n}\lambda^\delta\, dE_a(\lambda)\,,
$$
we have
\begin{equation}
\label{bound-difference}
\|a^\delta\, (1-\vf_n(a))\|\leq \|a^\delta\, (1-\vf_n^2(a))\|\leq\|a^\delta (1-P_n)\|\leq n^{-\delta},\quad \forall \delta>0\,.
\end{equation}

Finally, from $P_n\leq \vf_n(a)\leq P_{n+1}$, we deduce that 
\begin{equation}
\label{order}
\tfrac1n P_n\leq a_n\leq a^\vf_n\leq a_{n+1}\,,
\end{equation}
and that
\begin{equation}
\label{support}
\vf_n(a)P_{n+1}=\vf_n(a)\,,\quad \vf_n(a)P_n=P_n\,.
\end{equation}

The reason why two approximations are required is as follows. The projection based method
allows the use of several operator inequalities, most notably \cite[Lemma 3.3 (ii)]{CPS2}.
If we were then willing to assume that $[P_n,G]\in Z^0_1$, then the following proof would
simplify considerably. However, this assumption is highly implausible in the examples. So
we introduce the smooth approximation scheme, and a more complex proof, in order 
to obtain a result which is actually applicable to the examples.

The following is our main technical result from which Theorem \ref{3.8} will follow easily. 
\begin{prop}
\label{trace-diffrence}
Let $0\leq a\in \cn$ be such that there exists $\delta>0$ with $a^{1-\delta}\in B_\zeta$  and $[G,a^{1-\delta}]\in\Z_1^0$. Then
$$
\lim_{s\searrow 1}(s-1)\,\big\|aG^sa-(aGa)^s\big\|_1=0\,.
$$
\end{prop}
The proof of the proposition proceeds by writing
\begin{align*}
&aG^sa-(aGa)^s=\Big[aG^sa-a^\vf_nG^sa^\vf_n\Big]+\Big[a^\vf_nG^sa^\vf_n-a\big(\vf_n(a)G\vf_n(a)\big)^sa\Big]\\
&\,+\Big[a\big(\vf_n(a)G\vf_n(a)\big)^sa-a\big(P_nGP_n\big)^sa\Big]
+\Big[a\big(P_nGP_n\big)^sa-\big(a_nGa_n\big)^s\Big]
+\Big[\big(a_nGa_n\big)^s-\big(aGa\big)^s\Big]
\end{align*}
and then controlling each successive difference in this equality in the trace norm.
The following sequence of lemmas achieves this goal.

\begin{lemma}
\label{diff1}
Let $0\leq a\in \cn$ such that there exists $\delta>0$ with $a^{1-\delta}\in B_\zeta$. Then 
$$
\limsup_{s\searrow 1}
(s-1)\left\|aG^sa- a_n^\vf G^s  a_n^\vf\right\|_1\leq \Big(n^{-2\delta}+2\|a\|^{\delta}n^{-\delta}\Big)\,\|a^{1-\delta}\|_\zeta^2\,\,.
$$
\end{lemma}
\begin{proof}
Since $0\leq {a^\vf_n}^2\leq a^2$ it follows from Lemma \ref{Bp-grew}, $ii)$, that $ a_n^\vf\in B_\zeta$ and that the function 
$s\mapsto (s-1)\tau\big( a_n^\vf G^s a_n^\vf\big)$, for  $s\geq 1$, 
is well defined and bounded. Using the equality
$$
 a_n^\vf G^s  a_n^\vf-aG^sa=(1-\vf_n(a))aG^sa(1-\vf_n(a))-aG^sa(1-\vf_n(a))-(1-\vf_n(a))aG^sa\,,
$$
we obtain
\begin{align*}
\|aG^sa- a_n^\vf G^s  a_n^\vf\|_1\leq \Big(\|(1-\vf_n(a))a^\delta\|^2+2\|(1-\vf_n(a))a^{\delta}\|\|a^\delta\|\Big)\|a^{1-\delta}G^sa^{1-\delta}\|_1\,.
\end{align*}
Since $\|a^{\delta}(1-\vf_n(a))\|\leq  n^{-\delta}$, by equation \eqref{bound-difference},
we immediately  deduce the result.
\end{proof}

\begin{lemma}
\label{diff5}
Let $0\leq a\in\cn$ such that there exists $\delta>0$ with $a^{1-\delta}\in B_\zeta$. Then there exists two  constants $C_1,C_2>0$, uniform in $n$, such that
\begin{align*}
 i)&\qquad\limsup_{s\searrow 1}\,(s-1)\big\|(a_{n+1}Ga_{n+1})^s-(a_nGa_n)^s\big\|_1\leq 
C_1\,n^{-\delta/2}\,,\\
 ii)&\qquad\limsup_{s\searrow 1}\,(s-1)\big\|(a_n Ga_n)^s-(aGa)^s\big\|_1\leq 
C_2\,n^{-\delta/2}\,.
\end{align*}
\end{lemma}

\begin{proof}
To prove $i)$, let $A_n:=a_{n+1}\,G\,a_{n+1}$ and $ B_n:=a_n\,G\,a_n$.
Then 
\begin{align*}
\|A_n^s-B_n^s\|_1=\|A_n^{s/2}(A_n^{s/2}-B_n^{s/2})-(A_n^{s/2}-B_n^{s/2})B_n^{s/2}\|_1
&\leq \big(\|A_n^{s/2}\|_2+\|B_n^{s/2}\|_2\Big)\|A_n^{s/2}-B_n^{s/2}\|_2\\
&\leq \big(\|A_n\|_s^{s/2}+\|B_n\|_s^{s/2}\Big)\|A_n-B_n\|_s^{s/2}\,,
\end{align*}
by the BKS inequality since $0<s/2<1$. Then
we use $\|A_n\|_s\leq \|aGa\|_s$ and $\|B_n\|_s\leq\|aGa\|_s$, together with
\begin{align*}
\|A_n-B_n\|_s&=\|a_{n+1}Ga(P_{n+1}-P_{n})-(P_n-P_{n+1})aGa_n\|_s
\leq 2\,\|a^{\delta}(P_{n+1}-P_{n})\|\,\|a^{1-\delta}Ga\|_s\,,
\end{align*}
to obtain
$$
\|A_n^s-B_n^s\|_1\leq 2^{s/2+1}\,\|a\|^{3\delta s/2}\,\|a^{\delta}(P_{n+1}-P_n)\|^{s/2}\,\|a^{1-\delta}Ga^{1-\delta}\|_s^s\,.
$$
This concludes the proof since $(s-1)\|a^{1-\delta}Ga^{1-\delta}\|_s^s$ is bounded  and 
$\|a^{\delta}(P_{n+1}-P_n)\|^{s/2}\leq n^{-s\delta/2}$.

To prove $ii)$, one uses the same strategy applied to $A_n=a_n Ga_n$ and $B_n=aGa$.
\end{proof}

The following result is strongly inspired by \cite[Lemmas 3.3-3.5]{CPS2}:

\begin{lemma}
\label{PGP}
Let $P\in\cn$ be a projector and $0\leq a\in B_\zeta$ such that  $[a,P]=0$ and $a\geq m \,P$, for some $m\in (0,1)$. Then 
$$
\lim_{s\searrow 1}(s-1)\big\|a(PGP)^sa-(aP G Pa)^s\big\|_1=0\,.
$$
\end{lemma}

\begin{proof}
By \cite[Lemma 3.3 i)]{CPS2}, we have
$$
(aP G Pa)^s\leq\|a\|^{2(s-1)}\,a(PGP)^sa\,,
$$
The result follows if we can show that
\begin{equation}
\label{trick}
(aP G Pa)^s\geq m^{2(s-1)}\,a(PGP)^sa\,,
\end{equation}
as we would then have
\begin{equation}
\label{inequality}
\big( m^{2(s-1)}-1\big)a(PGP)^sa\leq(aP G Pa)^s-a(PGP)^sa\leq\big( \|a\|^{2(s-1)}-1\big)a(PGP)^sa\,,
\end{equation}
and this suffices by the following reasoning.
If $\|a\|\leq1$, then
$$
0\leq a(PGP)^sa-(aP G Pa)^s\leq \big( 1-m^{2(s-1)}\big)a(PGP)^sa\,,
$$
and the claim follows. So assume  $\|a\|>1$. Then, 
$$
-\big( \|a\|^{2(s-1)}-1\big)a(PGP)^sa\leq a(PGP)^sa-(aP G Pa)^s\leq \big( 1-m^{2(s-1)}\big)a(PGP)^sa\,.
$$
Setting
$$
0\leq b=\big( \|a\|^{2(s-1)}-1\big),\,0\leq c=\big( 1-m^{2(s-1)}\big),\,A=a(PGP)^sa,\, X=a(PGP)^sa-(aP G Pa)^s,
$$
we have $0\leq X+bA\leq (c+b) A$, and thus
$$
\|X\|_1\leq\|X+bA\|_1+b\|A\|_1\leq (c+2b)\|A\|_1\,,
$$
that is 
\begin{align*}
\|a(PGP)^sa-(aP G Pa)^s\big\|_1&\leq \Big(\big( 1-m^{2(s-1)}\big)+2\big( \|a\|^{2(s-1)}-1\big)\Big)\|a(PGP)^sa\|_1\,,
\end{align*}
which gives the result since 
$$
\|a(PGP)^sa\|_1\leq \|aPG^sPa\|_1=
\|PaG^saP\|_1\leq \|aG^sa\|_1\,,
$$
where we have used the operator inequality $a(PGP)^sa\leq aPG^sPa$,  from \cite[Lemma 3.3 i)]{CPS2}.

To prove \eqref{trick}, decompose $\cH$
as $P\cH\oplus (1-P)\cH$. Since $[P,a]=0$, we know that 
$$
(aPGPa)^s=P(aPGPa)^sP,\quad \mbox{and}\quad
 a(P GP)^sa=Pa(P GP)^saP,
$$
and so their restrictions to $(P\cH)^\perp_n$ are zero and so \eqref{trick} holds on 
$(1-P)\cH$. Since $a\geq m\, P$, its restriction to $P\cH$ is an invertible element of $P\cn P$  
and \cite[Lemma 3.3 ii)]{CPS2} gives the result.
\end{proof}

Next we prove some results involving both the projectors $P_n$ and their smooth versions
$\vf_n(a)$.

\begin{lemma}
\label{trivial}
Let $0\leq a\in\cn$ be such that there exists $\delta>0$ with $a^{1-\delta}\in B_\zeta$. Then 
there exists $C>0$, uniform in $n$, such that
$$
\limsup_{s\searrow 1}(s-1)\big\|a(P_nGP_n)^sa-a(P_{n+1}GP_{n+1})^sa\big\|_1
\leq C\,n^{-\delta/2}\,.
$$
\end{lemma}
\begin{proof}
Write
\begin{align*}
a(P_nGP_n)^sa-a(P_{n+1}GP_{n+1})^sa=\big(a(P_nGP_n)^sa-&(a_nGa_n)^s\big)+\big((a_nGa_n)^s-(a_{n+1}Ga_{n+1})^s\big)\\
&\qquad+\big((a_{n+1}Ga_{n+1})^s-a(P_{n+1}GP_{n+1})^sa\big)\,,
\end{align*}
and apply Lemma \ref{diff5} and Lemma \ref{PGP}.
\end{proof}

\begin{lemma}
\label{diff6}
Let $0\leq a\in \cn$ be such that there exists $\delta>0$ with  $a^{1-\delta}\in B_\zeta$. 
Then there exists $C>0$, uniform in $n$, such that
$$
\limsup_{s\searrow 1}\,(s-1)\,\big\|
a\big(\vf_n(a)G\vf_n(a)\big)^sa-a\big(P_nGP_n\big)^sa\big\|_1\leq C\, n^{-\delta/2}\,.
$$
\end{lemma}
\begin{proof}
By equation \eqref{support}, we have $\vf_n(a)=P_{n+1}\vf_n(a)$, while $P_n=\vf_n(a)P_n$.
Thus,
\begin{align*}
a\big(\vf_n(a)G\vf_n(a)\big)^sa-a\big(P_nGP_n\big)^sa&=
a\big(P_{n+1}\vf_n(a)G\vf_n(a)P_{n+1}\big)^sa-a\big(P_n\vf_n(a)G\vf_n(a)P_n\big)^sa\\
=&\Big[a\big(P_{n+1}\vf_n(a)G\vf_n(a)P_{n+1}\big)^sa-\big(aP_{n+1}\vf_n(a)G\vf_n(a)P_{n+1}a\big)^s\Big]\\
&\,\,+\Big[\big(aP_{n+1}\vf_n(a)G\vf_n(a)P_{n+1}a\big)^s-
\big(aP_n\vf_n(a)G\vf_n(a)P_na\big)^s\Big]\\
&\,\,+
\Big[\big(aP_n\vf_n(a)G\vf_n(a)P_na\big)^s-a\big(P_n\vf_n(a)G\vf_n(a)P_n\big)^sa\Big]\,.
\end{align*}
For the first term in parentheses, we can apply Lemma \ref{PGP}, with the  
modification that we replace $G$ there by $ \vf_n(a)G\vf_n(a)$, to obtain a 
vanishing contribution. Indeed, following line by line the proof of Lemma 
\ref{PGP} with the indicated modification, we get the operator inequalities
\begin{align*}
&\big( m^{2(s-1)}-1\big)a(P_{n+1}\vf_n(a)G\vf_n(a)P_{n+1})^sa\\
&\hspace{4cm}\leq(aP_{n+1}\vf_n(a)G\vf_n(a)P_{n+1}a)^s-
a(P_{n+1}\vf_n(a)G\vf_n(a)P_{n+1})^sa\\
&\hspace{8cm}\leq\big( \|a\|^{2(s-1)}-1\big)a(P_{n+1}\vf_n(a)G\vf_n(a)P_{n+1})^sa\,.
\end{align*}
Combining these operator inequalities  with $a(P_{n+1}\vf_n(a)G\vf_n(a)P_{n+1})^sa\leq aG^sa$, we obtain
$$
\lim_{s\searrow 1}(s-1)\|a\big(P_{n+1}\vf_n(a)G\vf_n(a)P_{n+1}\big)^sa-
\big(aP_{n+1}\vf_n(a)G\vf_n(a)P_{n+1}a\big)^s\|_1=0\,.
$$
Replacing $P_{n+1}$ by $ P_n$ gives the same conclusion for the last term in parentheses.

For the middle term, we can apply Lemma \ref{trivial} with the replacement 
$a\mapsto a^\vf_n$, to obtain the desired trace-norm bound. Indeed, 
since ${a_n^\vf}^2=a^2\vf_n^2(a)\leq a^2$, we infer from Lemma \ref{Bp-grew} 
ii) that $a^\vf_n\in B_\zeta$ and since  
$(a^\vf_n)^{1-\delta}=a^{1-\delta}\vf_n(a)^{1-\delta}\leq a^{1-\delta}$, 
we see that $(a^\vf_n)^{1-\delta}$ belongs to $B_\zeta$ too.
\end{proof}

To control the trace-norm of $a^\vf_nG^sa_n^\vf-a(\vf_n(a)G\vf_n(a))^sa$, we need  the following identity.

\begin{lemma}
\label{adam}
Let $0\leq A,B\in\cn$ with $ B$ injective. Then for all $1\leq s\leq 2$ we have the equality
$$
\big(ABA)^s=AB^{1/2}\big(B^{1/2}A^2B^{1/2}\big)^{s-1}B^{1/2} A\,.
$$
\end{lemma}
\begin{proof}
Let $B^{1/2}A=v\,|B^{1/2}A|$ so that $AB^{1/2}=|B^{1/2}A|\,v^*$,
where the partial isometry $v$ satisfies 
$$v:(\ker(B^{1/2}A))^\perp\to 
\overline{\mbox{range}(B^{1/2}A)}.$$
Moreover,
$$
{\rm support }(v^*v)=(\ker(B^{1/2}A))^\perp,\quad {\rm support }(vv^*)= \overline{\mbox{range}(B^{1/2}A)}\,.
$$
 Since $B$ is injective, we have $\ker(B^{1/2}A)=\ker(A)$,
and thus $v^*v={\rm support }(A)$.
Next, we remark that 
$$
\overline{\mbox{range}(B^{1/2}A)}^\perp=\ker(AB^{1/2})=\ker(|AB^{1/2}|)=\ker(B^{1/2}A^2B^{1/2})\,,
$$
and thus $vv^*$ is actually the projection onto the orthogonal complement of the kernel of 
$B^{1/2}A^2B^{1/2}$, so we immediately conclude that
\begin{align}
\label{bel}
vv^*B^{1/2}A^2B^{1/2}=B^{1/2}A^2B^{1/2}vv^*=B^{1/2}A^2B^{1/2}\,.
\end{align}
Observing that
\begin{align*}
AB^{1/2}=v^*vAB^{1/2}&=v^*v|B^{1/2}A|v^*=v^*B^{1/2}Av^*\quad {\rm and}\\
&\hspace{5cm}B^{1/2}A=B^{1/2}Av^*v=v|B^{1/2}A|v^*v=vAB^{1/2}v\,,
\end{align*}
we find
$$
ABA=v^*B^{1/2}Av^*vAB^{1/2}v=v^*B^{1/2}A^2B^{1/2}v.
$$
Thus restricting $ABA$ to $v^*v\cH$ (it is zero on $(1-v^*v)\cH$), we obtain
\begin{align*}
(ABA)^{s-1}&=\left(v^*B^{1/2}A^2B^{1/2}v\right)^{s-1}\\
&=C_{s-1}\int_0^\infty \lambda^{-s+1} v^*B^{1/2}A^2B^{1/2}v
(v^*v+\lambda v^* B^{1/2}A^2B^{1/2}v)^{-1}d\lambda\\
&=C_{s-1}\int_0^\infty \lambda^{-s+1}v^*B^{1/2}A^2B^{1/2}(vv^*+\lambda B^{1/2}A^2B^{1/2})^{-1}
d\lambda v\\
&=v^*\left(B^{1/2}A^2B^{1/2}\right)^{s-1}v\,,
\end{align*}
where we used that $vv^*$ is the support projection of $B^{1/2}A^2B^{1/2}$. We conclude by 
\begin{align*}
(ABA)^s&=|B^{1/2}A|\,(ABA)^{s-1}|B^{1/2}A|\\&
=|B^{1/2}A|v^*\,\left(B^{1/2}A^2B^{1/2}\right)^{s-1}\,v|B^{1/2}A|=AB^{1/2}\left(B^{1/2}A^2B^{1/2}\right)^{s-1}B^{1/2}A\,.\hspace{1,5cm}\qed
\end{align*}
\hideqed
\end{proof}

\begin{lemma}
\label{diff3}
Let $0\leq a\in \cn$ and suppose that there exists $\delta>0$ with  
$a^{1-\delta}\in B_\zeta$ and $[G,a^{1-\delta}]\in\Z_1^0$. 
Then there exists an absolute constant $C>0$ such that
$$
\limsup_{s\searrow 1}\,(s-1)\,\big\|a_n^\vf G^sa_n^\vf-a(\varphi_n(a)G\varphi_n(a))^sa\big\|_1\leq C \,n^{-\delta/2}\,.
$$
\end{lemma}
\begin{proof}
We remark first   from Lemma \ref{Bp-grew} i), that  $a^{1-\delta}\in B_\zeta$  implies  
 $a\in B_\zeta$ too. Then,
as $\vf_n(a)\leq P_{n+1}\leq (n+1) a_{n+1}$, we readily see that $\vf_n(a) G \vf_n(a)\in \Z_1$ and thus $(\vf_n(a) G \vf_n(a))^s$ is trace-class for all $s>1$, so that 
we are entitled to take the trace norm as in the statement of the lemma.

Lemma \ref{adam} applied to $A=\vf_n(a)$ and $B=G$ (which is injective), gives
\begin{align*}
&a_n^\vf G^sa_n^\vf-a(\varphi_n(a)G\varphi_n(a))^sa=
a_n^\vf G^{1/2}\big(G^{s-1}-(G^{1/2}\varphi_n(a)^2G^{1/2})^{s-1}\big)G^{1/2}a_n^\vf\\
&=a_n^\vf G^{1/2} \frac{\sin(\pi\eps)}\pi\int_0^\infty \lambda^{-\eps}\Big(G(1+\lambda G)^{-1}- {G^{1/2}\varphi_n(a)^2G^{1/2}}(1+\lambda G^{1/2}\varphi_n(a)^2G^{1/2})^{-1}
\Big)d\lambda G^{1/2}a_n^\vf\\
&=a_n^\vf G^{1/2} \frac{\sin(\pi\eps)}\pi\int_0^\infty \lambda^{-\eps}(1+\lambda G)^{-1}G^{1/2}\big(1-\vf_n^2(a)\big)G^{1/2}(1+\lambda G^{1/2}\varphi_n(a)^2G^{1/2})^{-1}
d\lambda G^{1/2}a_n^\vf\,,
\end{align*}
where we have defined $\eps:=s-1$ and  we have used in the third equality, the identity 
$$
A(1+\lambda A)^{-1}-B(1+\lambda B)^{-1}=(1+\lambda A)^{-1}(A-B)(1+\lambda B)^{-1}\,,\quad0\leq A,B\in\cn,\lambda>0\,.
$$
Hence
$$
a_n^\vf G^sa_n^\vf-a(\varphi_n(a)G\varphi_n(a))^sa=\frac{\sin(\pi\eps)}\pi\int_0^\infty X_{\eps,n}(\lambda)\,\lambda^{-\eps}d\lambda\,,
$$
with
$$
X_{\eps,n}(\lambda)=a_n^\vf G(1+\lambda G)^{-1}\big(1-\vf_n^2(a)\big)
G^{1/2}(1+\lambda G^{1/2}\varphi_n(a)^2G^{1/2})^{-1} G^{1/2}a_n^\vf\,.
$$

Commuting $1-\vf_n^2(a)$ with $G(1+\lambda G)^{-1}$ on its left, we obtain
$X_{\eps,n}(\lambda)=Y_{\eps,n}(\lambda)+Z_{\eps,n}(\lambda)$, with
\begin{align*}
Y_{\eps,n}(\lambda)&:=a_n^\vf \big(1-\vf_n^2(a)\big) {G^{1/2}}(G^{-1}+\lambda )^{-1}
(1+\lambda G^{1/2}\varphi_n(a)^2G^{1/2})^{-1}G^{1/2}a_n^\vf\,,\\
Z_{\eps,n}(\lambda)&:=a_n^\vf(1+\lambda G)^{-1}\big[G,\vf_n^2(a)\big] (G^{-1}+\lambda )^{-1}\Big(1-\lambda \varphi_n(a)^2(G^{-1}+\lambda \varphi_n(a)^2)^{-1}\Big) a_n^\vf\,,
\end{align*}
where we have used the relations
$$
[G(1+\lambda G)^{-1},\big(1-\vf_n^2(a)\big)]=-\lambda^{-1}[ (1+\lambda G)^{-1},\vf_n^2(a)]=(1+\lambda G)^{-1}[G,\vf_n^2(a)](1+\lambda G)^{-1}\,,
$$
and
$$
G^{1/2}(1+\lambda G^{1/2}\varphi_n(a)^2G^{1/2})^{-1} G^{1/2}=
(G^{-1}+\lambda \varphi_n(a)^2)^{-1}=G(1-\lambda \varphi_n(a)^2(G^{-1}+\lambda \varphi_n(a)^2)^{-1})\,,
$$
from the identity
$$
(A+B)^{-1}=A^{-1}(1-B(A+B)^{-1})\,,\quad A\geq1\,,B\geq 0\,.
$$

For $\|Y_{\eps,n}(\lambda)\|_1$, we use the H\"older inequality to obtain  the upper bound
$$
\|a^{\delta/2} \big(1-\vf_n^2(a)\big)\|\|a^{1-\delta/2}G^{(1+\eps)/2}\|_{\frac{2+\eps}{1+\eps}}\| G^{-\eps/2}(G^{-1}+\lambda )^{-1}\|\|
(1+\lambda G^{1/2}\varphi_n(a)^2G^{1/2})^{-1}\| \|G^{1/2}a\|_{2+\eps}$$
$$
\quad\leq n^{-\delta/2}\,(1+\lambda)^{-1+\eps/2}\|a^{1-\delta/2}G^{(1+\eps)/2}\|_{\frac{2+\eps}{1+\eps}}
\|G^{1/2}a\|_{2+\eps}\,,
$$
where we have used equation \eqref{bound-difference}  and
$$
G^{-\eps/2}({G^{-1}+\lambda })^{-1}\leq(G^{-1}+\lambda)^{-1+\eps/2} \leq
{(1+\lambda)^{-1+\eps/2} }\,.
$$
Next, from the operator inequality, \cite[Lemma 3.3 i)]{CPS2}, 
$(aGa)^{1+\eps/2}\leq \|a\|^\eps aG^{1+\eps/2}a$,
we obtain
\begin{align}
\label{estim}
\|G^{1/2}a\|_{2+\eps}&=\|aGa\|_{1+\eps/2}^{1/2}\nonumber\\
&\leq\|a\|^{\eps/(2+\eps)}\|aG^{1+\eps/2}a\|_1^{(2+\eps)^{-1}}
\leq \|a\|^{\eps/(2+\eps)}\|a\|_{\zeta}^{2/(2+\eps)}\,(2/\eps)^{(2+\eps)^{-1}}\leq C\,\eps^{-1/2}\,.
\end{align}
To evaluate $\|a^{1-\delta/2}G^{(1+\eps)/2}\|_{(2+\eps)/(1+\eps)}$, we write
$$
a^{1-\delta/2}G^{(1+\eps)/2}=a^{\delta/2}G^{(1+\eps)/2}a^{1-\delta}+a^{\delta/2}\big[a^{1-\delta},G^{(1+\eps)/2}\big]\,.
$$
For the first term, we obtain
$$
\|a^{\delta/2}G^{(1+\eps)/2}a^{1-\delta}\|_{(2+\eps)/(1+\eps)}\leq
\|a^{\delta/2} G^{\eps/2}\|_{1+2/\eps}\|G^{1/2}a^{1-\delta}\|_{2+\eps}\leq C\, \eps^{-1/2}\,,
$$
where we used that $\|a^{\delta/2} G^{\eps/2}\|_{1+2/\eps}$ remains bounded when $\eps\to 0^+$, 
from Lemma \ref{limits} ii), and the estimate
of equation \eqref{estim} for the second part, since $a^{1-\delta}\in B_\zeta$ by assumption.
It remains to treat the commutator term, for which 
Lemma \ref{commutators-powers} gives us
$$
\big\|\big[a^{1-\delta},G^{(1+\eps)/2}\big]\big\|_{(2+\eps)/(1+\eps)}\leq C
\big\|\big[a^{1-\delta},G\big]\big\|_{1+\eps/2}^{(1+\eps)/2}\,.
$$
We conclude using Lemma \ref{limits} i) that $\eps^{1/2}\big\|\big[a^{1-\delta},G^{(1+\eps)/2}\big]\big\|_{(2+\eps)/(1+\eps)}\to 0$ when $\eps\to 0^+$.
Hence, we have shown that
$$
\limsup_{\eps\searrow 0}\,\eps\, \|Y_{\eps,n}(\lambda)\|_1\leq C\, n^{-\delta/2}\,(1+\lambda)^{-1+\eps/2}\,.
$$
It remains to treat $Z_{\eps,n}(\lambda)$ which we estimate in trace-norm as
\begin{align*}
\|Z_{\eps,n}(\lambda)\|_1&\leq \|a_n^\vf ({1+\lambda G})^{-1}[G,\vf_n^2(a)] G({1+\lambda G})^{-1}a_n^\vf\|_1\\
&\hspace{2cm}+\|a_n^\vf (1+\lambda G)^{-1}[G,\vf_n^2(a)]G({1+\lambda G})^{-1}\lambda\vf_n^2(a)({G^{-1}+\lambda \varphi_n(a)^2})^{-1}a_n^\vf\|_1\,.
\end{align*}
We estimate the first term by
\begin{align*}
&\|a\|\,\|[G,\vf_n^2(a)]\|_{1+\eps/2}\| { G^{1-\eps/2}}({1+\lambda G})^{-1}\|\|G^{\eps/2}a\|_{1+2/\eps}\\
&\hspace{7cm}\leq \|a\|\,\|[G,\vf_n^2(a)]\|_{1+\eps/2}\|G^{\eps/2}a\|_{1+2/\eps}(1+\lambda)^{-1+\eps/2}\,.
\end{align*}
For the second term, we obtain the bound
\begin{align*}
\|a\|^2\,\big\|\big[G,\vf_n^2(a)\big]\big\|_{1+\eps/2}\|G^{\eps/2}\vf_n(a)\|_{1+2/\eps}\Big\|\vf_n(a)\frac \lambda{G^{-1}+\lambda \varphi_n(a)^2}\vf_n(a)\Big\|(1+\lambda)^{-1+\eps/2}\,.
\end{align*}
Since $G^{-1}\geq 1$, we have the estimate
$$
(G^{-1}+\lambda \varphi_n(a)^2)^{-1}\leq ({1+\lambda \varphi_n(a)^2})^{-1}\,,
$$
and thus
$$
\|\vf_n(a) \lambda({G^{-1}+\lambda \varphi_n(a)^2})^{-1}\vf_n(a)\|\leq
\| { \lambda\vf_n^2(a) }({1+\lambda \varphi_n(a)^2})^{-1}\|\leq 1\,.
$$
Using $\vf_n(a)\leq n\,a$, we obtain $\|G^{\eps/2}\vf_n(a)\|_{1+2/\eps}\leq n\|G^{\eps/2}a\|_{1+2/\eps}$ and so
\begin{align*}
\|Z_{\eps,\lambda}^n\|_1&\leq C(1+n)\big\|\big[G,\vf_n^2(a)\big]\big\|_{1+\eps/2}\|G^{\eps/2}a\|_{1+2/\eps}(1+\lambda)^{-1+\eps/2}\\
&\leq C(1+n)\|\widehat{{\vf_n^2}'}\|_1\big\|\big[G,a\big]\big\|_{1+\eps/2}\|G^{\eps/2}a\|_{1+2/\eps}(1+\lambda)^{-1+\eps/2}\,.
\end{align*}
We have used Lemma \ref{not-uniform} to obtain the last inequality. We stress that $\|\widehat{{\vf_n^2}'}\|_1$ is not uniform in $n$ since $\vf_n^2$ pointwise-converges to a step function.
However, combining Theorem 3.1 from \cite{PS1} with Theorem 4 from \cite{PS2}, 
and taking into account that $\Z_1^0$ is an interpolation space for the couple $(\L^1,\cn)$, we get from 
$[G,a^{1-\delta}]\in\Z_1^0$ that $[G,a]\in\Z_1^0$ as well.
Thus,  by Lemma \ref{limits} i), ii), we know that $\eps\|[G,a]\|_{1+\eps}\to 0$, while 
$\|G^{\eps}a\|_{1+1/\eps}$ remains bounded when $\eps\to 0^+$.
Putting everything together, we obtain the announced result:
\begin{align*}
&\limsup_{s\searrow 1}\,(s-1)\,\big\|a_n^\vf G^sa_n^\vf-a(\varphi_n(a)G\varphi_n(a))^sa\big\|_1\\
&\hspace{4cm}\leq
\limsup_{\eps\searrow 0}\,\eps\,\frac{\sin(\pi\eps)}\pi\int_0^\infty \big(\|Y_{\eps,n}(\lambda)\|_1+\|Z_{\eps,n}(\lambda)\|_1\big)\,\lambda^{-\eps}d\lambda\leq C n^{-\delta/2}\,.\qed
\end{align*}
\hideqed
\end{proof}
We are now ready to complete the proof of our main technical result.
\begin{proof}[Proof of Proposition \ref{trace-diffrence}]
We write:
\begin{align*}
&aG^sa-(aGa)^s=\Big[aG^sa-a^\vf_nG^sa^\vf_n\Big]+\Big[a^\vf_nG^sa^\vf_n-a\big(\vf_n(a)G\vf_n(a)\big)^sa\Big]\\
&\,+\Big[a\big(\vf_n(a)G\vf_n(a)\big)^sa-a\big(P_nGP_n\big)^sa\Big]
+\Big[a\big(P_nGP_n\big)^sa-\big(a_nGa_n\big)^s\Big]
+\Big[\big(a_nGa_n\big)^s-\big(aGa\big)^s\Big]\,.
\end{align*}
The $\limsup$, $s\to 1^+$ of the trace norm of the first bracket multiplied by $(s-1)$, is bounded  by $n^{-\delta}$ by Lemma \ref{diff1}, the second is bounded by $n^{-\delta/2}$ by Lemma \ref{diff3}, the third is bounded by $n^{-\delta/2}$ by Lemma \ref{diff6}, the fourth is bounded by $0$ by Lemma \ref{PGP} and the fifth by $n^{-\delta/2}$ by Lemma \ref{diff5} ii). This concludes the proof since it implies:
$\limsup_{s\to 1+}\big\|aG^sa-(aGa)^s\big\|_1\leq C\,n^{-\delta/2}$, $\forall n\in\mathbb N$.
\end{proof}

Proposition \ref{trace-diffrence} immediately gives us 

\begin{corollary}\label{3.6}
Assume that $0\leq a\in \cn$ is such that there exists $\delta>0$ with $a^{1-\delta}\in B_\zeta$  and $[G,a^{1-\delta}]\in\Z_1^0$.  Then\\
(i) $ \lim_{s\to 1^+}(s-1)\tau\big(aG^sa\big) $ exists if and only if
$\lim_{s\to 1^+}(s-1)\tau\big((a Ga)^s\big)$ exists and then they are equal;\\
(ii) More generally, for any  Banach  limits 
${\omega}$, we have
$$
\tilde\omega-\lim_{s\searrow 1}(s-1)\tau\big(aG^sa\big) =
\tilde\omega-\lim_{s\searrow 1}(s-1)\tau\big((aGa)^s\big)\,.
$$
\end{corollary}

\subsection{Dixmier-traces computation}

We have arrived at the following analogue of \cite[Corollary
3.7]{CPS2}.
\begin{prop}\label{3.7}
Let $0\leq a\in\cn$ be such that there exists $\delta>0$ with  $a^{1-\delta}\in B_\zeta$
and $[G,a^{1-\delta}]\in\Z_1^0$. Then, if any one of the following limits exist
they all do and all coincide\\
(1) $\lim_{t\to\infty}\frac{1}{\log(1+t)}\int_0^t\mu_s(aGa) ds$,\\
(2) $\lim_{r\to\infty}\frac{1}{r}\,\tau((aGa)^{1+\frac{1}{r}})$,\\
(3) $\lim_{r\to \infty}\frac{1}{r}\,\zeta(a,a;1+\tfrac1r)$,\\
(4) $\lim_{\lambda\to\infty}\big(Mg(a,a;.)\big)(\lambda)$.\\
Furthermore, the existence of any of the above limits is
equivalent to\\
(5) every generalized limit $\omega$ which is dilation invariant
yields the same value $\tau_\omega(aGa)$
and the latter value coincides with the value of the limits above.
\end{prop}
\begin{proof}
The simultaneous existence and equality of (2) and (3) (resp. (3) and (4))
 follows from Corollary \ref{3.6} (resp. Theorem \ref{OL}). Recall now that the assumption
$a\in B_\zeta$ guarantees $aGa\in \Z_1$.
The assertion \lq\lq  (2) exists if and only if  (1) exists
and then they are equal" is known, it follows e.g.  by the same
argument as in the proof of \cite[Corollary 3.7]{CPS2} or by the
argument given at the beginning of the proof of \cite[Theorem
2]{BeF}. If (1) exists then it equals
$\tau_\omega(a^*Ga)$ by definition. The equivalence of
(4) and the existence of the limit (1) and their equality follows
from the main result of \cite{LSS}.
\end{proof}

We have arrived at the following nonunital analogue of
\cite[Theorem 3.8]{CPS2}, \cite[Theorem 4.11]{CRSS} and
\cite[Corollary 3.3]{LS}.

\begin{theorem}\label{3.8}
Assume that $a\in \cn$ is self adjoint and let $a=a_+-a_-$
be the decomposition into the difference of nonnegative operators. 
Assume that there exists $\delta>0$ with $a_{\pm}^{1/2-\delta}\in B_\zeta$,  
and $[G,a_\pm^{1/2-\delta}]\in\Z_1^0$. Then $aG\in\Z_1$, and moreover,\\
 (i) if $\lim_{s\to 1^+}(s-1)\big(\zeta(a_+^{1/2},a_+^{1/2};s)-\zeta(a_-^{1/2},a_-^{1/2};s)\big)$
exists, then it is equal to $\tau_\omega(aG)$ where we choose
$\omega$ as in
the proof of \cite[Theorem 4.11]{CRSS},\\
(ii) more generally, if we choose functionals $\omega$ and
$\tilde\omega$ as in the proof of \cite[Theorem 4.11]{CRSS}, then
$$
\tau_\omega(aG)=\tilde\omega-\lim_{r\to\infty}
\frac{1}{r}\big(\zeta(a_+^{1/2},a_+^{1/2};1+\tfrac1r)-\zeta(a_-^{1/2},a_-^{1/2};1+\tfrac1r)\big).
$$
\end{theorem}

\begin{proof} 
As observed earlier, recall that $a_{\pm}^{1/2-\delta}\in B_\zeta$ implies 
$a_{\pm}^{1/2}\in B_\zeta$
and $[G,a_\pm^{1/2-\delta}]\in\Z_1^0$ implies $[G,a_\pm^{1/2}]\in\Z_1^0$. 
Thus
$$
aG=a_+G-a_-G=a_+^{1/2}Ga_+^{1/2}-a_-^{1/2}Ga_-^{1/2}
+a_+^{1/2}[a_+^{1/2},G]-a_-^{1/2}[a_-^{1/2},G]\in\Z_1,
$$
with $\tau_\omega(aG)=\tau_\omega(a_+^{1/2}Ga_+^{1/2})
-\tau_\omega(a_-^{1/2}Ga_-^{1/2})$, where 
$\omega$ (and latter 
$\tilde\omega$) has been chosen as in the proof of 
\cite[Theorem 4.11]{CRSS}. We only need to prove part ii); part i) 
will then follow from general facts on Banach limits.
By \cite[Lemma
3.2(i)]{CPS2}, \cite[Theorem 4.11]{CRSS}, Proposition
\ref{3.6} and the remark above, we have
\begin{align*}
\tau_\omega(aG)=\tau_\omega(a_+^{1/2}Ga_+^{1/2})
-\tau_\omega(a_-^{1/2}Ga_-^{1/2})
&=\tilde\omega-\lim_{r\to\infty}\frac{1}{r}\big(\tau((a_+^{1/2}Ga_+^{1/2})^{1+\frac{1}{r}})
-\tau((a_-^{1/2}Ga_-^{1/2})^{1+\frac{1}{r}})\big)\\
&=\tilde\omega-\lim_{r\to\infty}\frac{1}{r}\big(\tau(a_+^{1/2}G^{1+\frac{1}{r}}a_+^{1/2})
-\tau(a_-^{1/2}G^{1+\frac{1}{r}}a_-^{1/2})\big),
\end{align*}
which concludes the proof.
\end{proof}

By considering independently real and imaginary parts of $a\in B_\zeta$, we 
get an analogous result for non-self-adjoint elements. Moreover, we could have 
stated a similar result using the C\'esaro mean of the heat-trace function instead 
of the zeta function. Namely, under the same assumptions as those of 
Theorem \ref{3.8}, it is true that
$$
\tau_\omega(aG)=
\omega-\lim_{\lambda\to\infty}
\big(\big(Mg(a_+^{1/2},a_+^{1/2};.)\big)(\lambda)
-\big(Mg(a_-^{1/2},a_-^{1/2};.)\big)(\lambda)\big).
$$

We close this Section with

\begin{prop} 
\label{prop:dixy-trace}
Let  $a\in B_\zeta$ be such that 
$[G,a]\in\Z_1^0$. Then for any Dixmier trace 
$$
0\leq\tau_\omega(a^*aG)=\tau_\omega(aa^*G).
$$
\end{prop}

\begin{proof}
By Theorem \ref{thm:Z_1-landing}, $a^*Ga,aGa^*\in\cz_1$, and since $aa^*G=aGa^*+a[a^*,G]$, $a^*aG=a^*Ga+a^*[a,G]$, $aa^*G$ and $a^*aG$ belong to $\cz_1$ as well.

Now, for $A\in\Z_2$, we have
$$
\tau_\omega(A^*A)=\omega-\lim_{t\to +\infty}\frac{\int_0^t\mu_s(A)^2ds}{\log(1+t)}=
\omega-\lim_{t\to +\infty}\frac{\int_0^t\mu_s(A^*)^2ds}{\log(1+t)}=\tau_\omega(AA^*)\,.
$$
Now we do some rearranging
\begin{align*}
\tau_\omega(a^*aG)&=\tau_\omega(a^*Ga)+\tau_\omega(a^*[a,G])
=\tau_\omega(a^*Ga)=\tau_\omega(G^{1/2}aa^*G^{1/2})\\
&=\tau_\omega(aG^{1/2}a^*G^{1/2})+\tau_\omega([G^{1/2},a]a^*G^{1/2})
=\tau_\omega(aa^*G)+\tau_\omega(a[G^{1/2},a^*]G^{1/2}),
\end{align*}
as the Dixmier trace vanishes on the ideal $\Z_1^0$, by Lemma \ref{commutators-powers}
$[G^{1/2},a]\in\Z_2^0$, and the fact that  $\Z_2^0\,\Z_2\subset\Z_1^0$.
We complete the proof by observing that
 $$
 a[G^{1/2},a^*]G^{1/2}= a[G,a^*]-aG^{1/2}[G^{1/2},a^*],
 $$ 
 and that both terms on the right hand side of this last equation are in $\Z_1^0$. 
\end{proof}

\section{The converse estimate}
\label{conter-ex}
In the previous Section we have shown that 
$
a\in B_\zeta\Rightarrow a^*Ga\in \Z_1\,.
$
 But the converse requires more assumptions.
 We demonstrate this by exhibiting a counter-example.
Let us introduce one more Banach  $*$-sub-algebra of $\cn$:
\begin{align*}
B_{\Z_1}=B_{\Z_1}(G)&:=\big\{a\in\cn:\|a^*Ga\|_{1,\infty}+\|aGa^*\|_{1,\infty}+\|a\|^2<\infty \big\}.
\end{align*}
It is easy to see that in general $a\in B_{\Z_1}\not\Rightarrow a\in B_\zeta$ by considering the case $G=Id_\cn$.
Then $B_{\Z_1}=\Z_2$ and $B_\zeta =\L^2$.
More realistic examples from spectral triples are also easy to produce.
A positive result in the converse direction is the following.
\begin{prop}
\label{converse}
Let $0\leq a\in \cn$, be such that there exists $\delta>0$ with $a^{1-\delta}\in B_{\Z_1}$ and $[G,a^{1-\delta}]\in\Z_1^0$. Then $a\in B_\zeta$.
\end{prop}

The proof of this `converse' estimate relies essentially on the following lemma.

\begin{lemma}
\label{conbou}
Let $0\leq a\in B_{\Z_1}$ and $\delta \in(0,1)$. Then, the map 
$\eps\in(0,\delta/2] \mapsto \|a^\delta G^\eps\|_{1+1/\eps}$, 
is bounded.
\end{lemma}
\begin{proof}
We know by assumption that $aGa\in\Z_1$ and thus from Proposition  \ref{complex-interpol}, we know that for $\delta \in(0,1)$, $a^\delta G^\eps\in\Z_{1/\eps}$ with
$$
\|a^\delta G^\eps\|_{1/\eps,\infty}\leq  \|a\|^{\delta-2\eps}\|aGa\|_{1,\infty}^\eps\,.
$$
We conclude using the same chain of estimates as in the proof of Lemma \ref{limits} ii):
\begin{align*}
\|a^\delta G^\eps\|_{1+1/\eps}=\||a^\delta G^\eps|^{1/\eps}\|_{1+\eps}^{\eps}
&\leq \Big(  \eps^{-\frac1{1+\eps}} \||a^\delta G^\eps|^{1/\eps}\|_{1,\infty}\Big)^{\eps}\\
&\hspace{2.5cm}=  \eps^{-\frac\eps{1+\eps}}\|a^\delta G^\eps\|_{{1/\eps},\infty}
 \leq  \eps^{-\frac\eps{1+\eps}} \|a\|^{\delta-2\eps}\|aGa\|_{1,\infty}^\eps.\qed
\end{align*}
\hideqed
\end{proof}

\begin{proof}[Proof of Proposition \ref{converse}]
Note that by the stability of $B_\zeta$ and  $B_{\Z_1}$ under the map $a\mapsto |a|$, we may assume
$a\geq 0$.
Then, we write
$$
aG^{1+\eps}a=a^\delta G^{1/2+\eps}a^{1-\delta}G^{1/2}a+a^\delta [a^{1-\delta},G^{1/2+\eps}]G^{1/2}a
\,,$$
and thus
\begin{align}
\label{hoo}
&\|aG^{1+\eps}a\|_1
\leq\|a^\delta G^\eps\|_{1+1/\eps}\|G^{1/2}a^{1-\delta}\|_{2+2\eps}\|G^{1/2}a\|_{2+2\eps}\\
&\hspace{5cm}+\|a\|^\delta\|[a^{1-\delta},G^{1/2+\eps}]\|_{2(1+\eps)/(1+2\eps)}\|G^{1/2}a\|_{2+2\eps}
\nonumber\,.
\end{align}
From Lemma \ref{commutators-powers}, we have for $\eps\leq 1/2$
\begin{equation}
\label{RTE}
\|[a^{1-\delta},G^{1/2+\eps}]\|_{2(1+\eps)/(1+2\eps)}\leq \|a\|^{(1-\delta)(1/2-\eps)}\|[a^{1-\delta},G]\|_{1+\eps }^{1/2+\eps}\,.
\end{equation}
Since $[a^{1-\delta},G]\in\Z_1^0$, we know from Lemma \ref{limits} i) that 
$\|[a^{1-\delta},G]\|_{1+\eps }=o(\eps^{-1})$. In particular, $\|[a^{1-\delta},G]\|_{1+\eps }^\eps=O(1)$ and thus the inequality \eqref{RTE} gives $\|[a^{1-\delta},G^{1/2+\eps}]\|_{2(1+\eps)/(1+2\eps)}=o(\eps^{-1/2})$.
Moreover, since $a^{1-\delta}Ga^{1-\delta}\in \Z_1$, we know from Theorem \ref{great}
that $\|a^{1-\delta}Ga^{1-\delta}\|_{1+\eps}=O(\eps^{-1})$, which gives
$$
\|G^{1/2}a\|_{2+2\eps}\leq \|a\|^\delta \|G^{1/2}a^{1-\delta}\|_{2+2\eps}=
\|a\|^\delta \|a^{1-\delta}Ga^{1-\delta}\|_{1+1\eps}^{1/2}=O(\eps^{-1/2})\,.
$$
Finally, by Lemma \ref{conbou} we know that
$ \|a^\delta G^\eps\|_{1+1/\eps}=O(1)$. Putting everything together,
the inequality \eqref{hoo} gives
us $\|aG^{1+\eps}a\|_1=O(\eps^{-1})$, that is
$\sup_{\eps> 0}\eps\|aG^{1+\eps}a\|_1<\infty$,
i.e. $a\in B_\zeta$.
\end{proof}

Now we can state a $\Z_1$ version of Proposition \ref{trace-diffrence}.

\begin{prop}
\label{trace-diffrence2}
Let $0\leq a\in \cn$ be such that there exists $\delta>0$ with $a^{1-\delta}\in B_{\Z_1}$,  
and $[G,a^{1-\delta}]\in\Z_1^0$. Then
$$
\lim_{s\to 1^+}(s-1)\,\big\|aG^sa-(aGa)^s\big\|_1=0\,.
$$
\end{prop}
\begin{proof}
From Proposition \ref{converse}, we  deduce that  $a^{1-\delta/2}\in B_{\zeta}$. 
Combining Theorem 3.1 from \cite{PS1} with Theorem 4 from \cite{PS2}, 
and taking into account that $\Z_1^0$ is an interpolation space for the couple 
$(\L^1,\cn)$, we get from 
$[G,a^{1-\delta}]\in\Z_1^0$ that $[G,a]\in\Z_1^0$ as well. Then
the claim follows directly from Proposition \ref{trace-diffrence}.
\end{proof}

\section{Nonunital Spectral Triples}
\label{sec:nonunital-case}

We will now use the 
results of the previous Sections to give an {\it a posteriori} definition of a finitely summable
nonunital semifinite spectral triple 
$(\A,\cH,\D)$, relative to $(\cn,\tau)$.
A semifinite spectral triple  consists of a separable Hilbert space 
$\cH$ carrying a faithful 
representation of $\cn$, together with an essentially self-adjoint operator $\D$ 
affiliated  with $\cn$ and a (nonunital) $*$-sub-algebra $\A$ of 
$\cn$ such that $[\D,a]$ is bounded for all $a\in\A$ and $a(1+\D^2)^{-1/2}\in\K_\cn$. 

The main difference between the notions of finitely summable 
{\em unital} and  {\em nonunital} spectral triple is that, 
in the unital case, $\D$ alone is enough to characterise the spectral dimension. 
However, the situation in the nonunital case is far more subtle  since one 
needs a delicate interplay between $\A$ and $\D$ to obtain a good definition  of
spectral dimension. 

Despite our previous focus on `$L^1$ as the square of $L^2$', 
we now define summability
for spectral triples in an `$L^1$' fashion. The reason for this is 
the local index
formula, which we will address elsewhere. However here we will 
quickly return, via the 
results of previous sections, to the $L^2$ type description.

\begin{definition}\label{def:spec-dim} 
Let $(\A,\cH,\D)$ be a nonunital semifinite spectral triple, relative to $(\cn,\tau)$.
We then let
$$
p:=\inf\{s>0\,:\,{\rm \ for\  all \ }0\leq a\in\A\,, \ \tau\big(a(1+\D^2)^{-s/2}\big)<\infty
\}\,,
$$
and when it exists, we say that $(\A,\HH,\D)$ is finitely summable and 
call $p$ the spectral dimension of the triple $(\A,\HH,\D)$.
\end{definition}

{\bf Remark}. This definition is closer in spirit to that employed in the local index formula,
and we will examine, and augment, this definition in another place. The use of 
the unsymmetrised form of the condition amounts to making the strongest possible 
assumption. We stress the dependence on the algebra $\A$ in the previous definition. 
Finally, note that by \cite{Bik} we obtain a positive functional for each $s>0$ since for $a\geq 0$,
$$
\tau(a(1+\D^2)^{-s/2})=\tau((1+\D^2)^{-s/4}a(1+\D^2)^{-s/4})\geq 0.
$$

\begin{definition}
Let $(\A,\cH,\D)$ be a finitely summable 
nonunital semifinite spectral triple, relative to $(\cn,\tau)$, with spectral dimension $p\geq 1$.
If  for all $a\in\A$ 
$$
\limsup_{s\searrow p}\left|(s-p)\tau(a(1+\D^2)^{-s/2})\right|<\infty
$$
we say that $(\A,\HH,\D)$
is $\Z_p$-summable. In this case we let $B_p:=B_\zeta((1+\D^2)^{-p/2})$.
\end{definition}

The special case described by the definition of $\Z_p$ is the nonunital analogue of
the most studied summability criteria in the unital case, usually called 
$(p,\infty)$-summability. Variations on the definition are certainly possible, and
only further examples can determine the best form of the definition.
Note that given $(\cn,\tau,\D)$, we have the inclusions $B_p\subset B_q$, $p\leq q$.

\begin{lemma} 
Let $(\A,\cH,\D)$ be a $\Z_p$-summable nonunital semifinite spectral triple, 
relative to $(\cn,\tau)$. Suppose that there is a $*$-algebra $\B\subseteq B_p$
such that
$\A\subset \B^2={\rm span}\{b_1b_2:b_1,\,b_2\in\B\}$  
and $b_1[b_2,(1+\D^2)^{-p/2}]\in\Z_1^0$ 
for all $b_1,\,b_2\in \B$. Then 
$$
a\mapsto\tau_\omega(a(1+\D^2)^{-p/2})
$$
defines a positive trace on $\A$.
\end{lemma}

\begin{proof}
We make our usual abbreviation $G=(1+\D^2)^{-p/2}$.
With the assumption above we have a finite sum $a=\sum_ib^i_1b^i_2$ and so
\begin{align*}
aG = \sum_i b^i_1b^i_2G
=\sum_i b^i_1[b^i_2,G]+b_1^iGb_2^i.
\end{align*}
Since for each $i$, $b^i_1[b^i_2,G]\in\Z^0_1$ and $b^i_1Gb^i_2\in\Z_1$, 
for any Dixmier trace $\tau_\omega$,
$$
\tau_\omega(aG)=\sum_i\tau_\omega(b^i_1Gb^i_2)
$$
is well-defined. To see that $a\mapsto \tau_\omega(aG)$ is an (unbounded) trace we employ
Proposition \ref{prop:dixy-trace}. We just need to check that 
$[a,G]\in Z^0_1$ for positive
$a\in\A$. However, using the Leibnitz rule
$$
[a,G]=\sum_ib^i_1[b^i_2,G]+[b^i_1,G]b^i_2\in Z_1^0.
$$
Positivity follows from \cite{Bik} again since $\tau_\omega(aG)=\tau_\omega(G^{1/2}aG^{1/2})\geq 0$
for $a\geq 0$.
\end{proof}

\begin{corollary}
\label{cor:compute-dixy}
Let $(\A,\cH,\D)$ be a $\Z_p$-summable nonunital semifinite spectral triple, 
relative to $(\cn,\tau)$. Suppose that $0\leq a\in\A$ is of the form 
$b^2$ for $b\geq 0$ with $b^{1-\delta}\in\B\subseteq B_p$ for some $\delta>0$, and with $b[b,G]\in\Z^0_1$. 
Then choosing 
$\omega,\,\tilde\omega$
as in \cite[Theorem 4.11]{CRSS} we have
$$
\tau_\omega(aG)= 
\tilde\omega-\lim\frac{1}{r}\tau(bG^{1+\frac{1}{r}}b)
$$
\end{corollary}

\begin{proof} This is the content of our main result Theorem \ref{3.8}.
\end{proof}

Our final aim is to  check the validity of our assumptions in the 
context of spectral triples, using smoothness of the spectral triple.

\begin{lemma}  In the definition of $B_p$ we may use $G=(1+\D^2)^{-p/2}$
interchangeably with  $G_1=(1+|\D|)^{-p}$.
\end{lemma}

\begin{proof} First note that we may apply Lemma \ref{Bp-grew}
in this situation. This is because we may write $G_1=Gf(G)$ or $G=G_1g(G_1)$
with $f,g\in L^\infty(\R)$. Thus in the definition of $B_p$ 
either $G$ or $G_1$ may be used. 
We note in addition that for the purpose
of studying commutators, we may also use $G$ or $G_1$
interchangeably as can be seen by the following sketch argument.
 For $a,b\in B_p$ we have
$a(G-G_1)b= aG(1-G_1G^{-1})b$. It is not difficult to 
find the bounded function $f$ with $G_1=f(G)$.
Now define a function $h$ such that
$1-f(G)G^{-1} = G^{1/p}h(G)$. Then some careful calculation proves explicitly that $h$ is bounded.
So we have $a(G-G_1)b= aG^{1/2}h(G)G^{1/2+1/p}b$ which is trace class.
\end{proof}

\begin{definition}
Let $(\A,\cH,\D)$ be a $\Z_p$-summable nonunital semifinite spectral triple, 
relative to $(\cn,\tau)$ with $\A\subseteq \B^2$, $\B\subseteq B_p$. 
We say $(\A,\HH,\D)$ 
is smooth to order $k$ if for all $b\in\B$ and for all $0\leq j\leq k$,  
$\delta^j(b)\in B_p$. Here $\delta$ is the unbounded derivation 
given by $\delta(T):=[|\D|,T]$, $T\in\cn$.
\end{definition}

Now we check that being smooth to order $\lfloor p\rfloor+1$ 
($\lfloor .\rfloor$ is the integer-part function) is enough to ensure 
that $b_1[b_2,(1+\D^2)^{-p/2}]\in\Z_1^0$ for $b_1,\,b_2\in\B\subseteq B_p$.

\begin{prop}
\label{TC}
Let  $(\A,\HH,\D)$ be a finitely summable nonunital semifinite spectral triple 
with spectral dimension $p$, and which is smooth to order $\lfloor p\rfloor+1$.
Suppose moreover that $\A\subseteq\B^2$ for $\B\subseteq B_p$.
Then for all $b_1,\,b_2\in\B\subseteq B_\zeta$,  
$b_1[b_2,(1+|\D|)^{-p}]$ is trace class.
\end{prop}

\begin{proof}
We may write
$$
(1+|\D|)^{-p}=\frac{1}{2\pi i}\int_\ell \lambda^{-p}(\lambda-1-|\D|)^{-1} d\lambda,
$$
where $\ell=\{a+iv : \  v\in \R,\, a=1/2\}$.
Then we have
\begin{align*}
[(1+|\D|)^{-p},b_2]&=\frac{1}{2\pi i}\int_\ell \lambda^{-p}(\lambda-1-|\D|)^{-1} 
[|\D|,b_2](\lambda-1-|\D|)^{-1}d\lambda\\
&\hspace{-1cm}=p(1+|\D|)^{-p-1}[|\D|,b_2]
-\frac{1}{2\pi i}\int_\ell \lambda^{-p}(\lambda-1-|\D|)^{-1} 
[(\lambda-1-|\D|)^{-1},[|\D|,b_2]]d\lambda.
\end{align*}
Repeat the previous resolvent trick to 
obtain
$$
[(1+|\D|)^{-p}\!,b_2]=p(1+|\D|)^{-p-1}[|\D|,b_2]
-\frac{1}{2\pi i}\int_\ell \lambda^{-p}(\lambda-1-|\D|)^{-2} 
[|\D|,[|\D|,b_2]](\lambda-1-|\D|)^{-1}d\lambda.
$$
Now iterate the process and multiply on the left by $b_1$ to obtain
\begin{align*}
&b_1[(1+|\D|)^{-p},b_2]=p\,b_1(1+|\D|)^{-p-1}[|\D|,b_2]-p(p+1)\,b_1
(1+|\D|)^{-p-2}[|\D|,[|\D|,b_2]]\ldots
\\
&\quad\quad\ldots
-\frac{1}{2\pi i}\int_\ell \lambda^{-p}\,b_1(\lambda-1-|\D|)^{-\lfloor p\rfloor-1} 
[|\D|, \ldots, [|\D|,[|\D|,b_2]]\cdots](\lambda-1-|\D|)^{-1}d\lambda,
\end{align*}
 where there are $\lfloor p\rfloor +1$ commutators 
in the integrand. 
Under the smoothness assumption  $\delta^k(b_2)\in B_p$ for 
$k\leq \lfloor p\rfloor+1$ we may use Lemma \ref{Bp-grew} to argue that for each 
$\lambda\in\ell$ 
the product
$$
b_1(\lambda-1-|\D|)^{-\lfloor p\rfloor-1} 
[|\D|, \ldots, [|\D|,[|\D|,b_2]]\cdots]
$$ 
is trace class because 
$b_1(1+|\D|)^{-\lfloor p\rfloor-1} [|\D|, \ldots, [|\D|,[|\D|,b_2]]\cdots]$ 
is trace class. Simple estimates now show that the integral converges in trace norm.
The result follows.
\end{proof}

{\bf Example}. Take $(\cn,\tau)=(\B(L^2(\R^p,S)),\mbox{Tr})$, where $S$ is the 
trivial spinor bundle, and $\mbox{Tr}$ the usual operator trace. We let 
$\dslash$ be the standard Dirac operator, $G=(1+\dslash^2)^{-p/2}$, 
$\HH=L^2(\R^p,S)$, and $\A$
the bounded smooth integrable functions all of whose partial derivatives are 
bounded and integrable.
Let $b\in L^2(\R^p)$ be the function 
$b(x)=(1+\Vert x\Vert^2)^{-p/4-\epsilon/2}$. 
Then $\A\ni a=b^2$
satisfies the hypotheses of Corollary \ref{cor:compute-dixy} and
\begin{align*}
\mbox{Tr}_{\omega}\left((1+\Vert x\Vert^2)^{-p/2-\epsilon}G\right)
&=\tilde\omega - \lim_{r\to\infty}\frac{1}{r}\mbox{Tr}
\left((1+\Vert x\Vert^2)^{-p/4-\epsilon/2}G^{1+\frac{1}{r}}(1+\Vert x\Vert^2)^{-p/4-\epsilon/2}\right)\\
&=C_p\int_{\R^p}\,(1+\Vert x\Vert^2)^{-p/2-\epsilon}\,d^px,
\end{align*}
the final line following from an explicit calculation similar to that in \cite[Corollary 14]{R2}.

Recall that an algebra $\A_c$ has local units when, for
any finite subset of elements $\{a_1,\dots,a_k\}$ of~$\A_c$, there
exists $u \in \A_c$ such that $ua_i = a_iu = a_i$ for $i = 1,\dots,k$. In the example above we 
could take $\A_c$ to be the smooth compactly supported functions, and apply the theory in
\cite{R2} to elements of $\A_c$.
However, the local units based theories of \cite{GGISV,R2} can not accomodate the
integration of the function $b^2$.
Indeed, since
the function $b^2$ is nowhere vanishing, the condition $b^2u=b^2$ implies necessarily 
that $u$ is the constant unit function, which does not belong to $\A$. Thus even in the classical
case of manifolds our approach to integration allows the integration of many more functions.
Similar examples may be constructed on any complete manifold, \cite{R2}.

\end{document}